\documentclass[a4paper,11pt,twoside]{article}

\usepackage{amsmath,amssymb,amsthm,mathrsfs}
\usepackage{hyperref}
\usepackage[left=1.45in,right=1.45in]{geometry}

\swapnumbers
\numberwithin{equation}{section}

\newtheorem{Thm}{Theorem}[section]
\newtheorem{Prop}[Thm]{Proposition}
\newtheorem{Lem}[Thm]{Lemma}
\theoremstyle{definition}
\newtheorem{Rem}[Thm]{Remark}
\newtheorem{Expl}[Thm]{Example}

\newcommand{\R}{\mathbb{R}}
\newcommand{\Z}{\mathbb{Z}}

\newcommand{\cA}{\mathscr{A}}
\newcommand{\cC}{\mathscr{C}}
\newcommand{\cF}{\mathscr{F}}
\newcommand{\cP}{\mathscr{P}}

\newcommand{\D}[1]{\Delta(#1)}
\newcommand{\Dl}[1]{\Delta_1(#1)}
\newcommand{\E}[1]{{\rm E}(#1)}
\newcommand{\e}{{\rm e}}
\newcommand{\Ex}[1]{{\rm E}'(#1)}

\newcommand{\bw}{I}
\newcommand{\co}{C}
\newcommand{\CAT}{\text{\rm CAT}}
\newcommand{\del}{\delta}
\newcommand{\diam}{\operatorname{diam}}
\newcommand{\eps}{\varepsilon}
\newcommand{\es}{\emptyset}
\newcommand{\Fix}{\operatorname{Fix}}
\newcommand{\Gam}{\Gamma}
\newcommand{\gam}{\gamma}
\newcommand{\Lam}{\Lambda}
\newcommand{\lam}{\lambda}
\renewcommand{\mid}{\,|\,}
\renewcommand{\H}{\text{\rm H}}
\newcommand{\id}{\operatorname{id}}

\newcommand{\Isom}{\operatorname{Isom}}
\newcommand{\la}{\langle}
\newcommand{\ra}{\rangle}
\newcommand{\Lip}{\operatorname{Lip}}
\newcommand{\ol}{\overline}
\renewcommand{\rho}{\varrho}
\newcommand{\rk}{\operatorname{rk}}

\newcommand{\sig}{\sigma}
\newcommand{\Sig}{\Sigma}
\newcommand{\sm}{\setminus}

\newcommand{\sub}{\subset}

\newcommand{\vc}[1]{[#1]}

\title{Injective hulls of certain discrete metric spaces\\and groups}
\author{Urs Lang\footnote{Department of Mathematics, ETH Z\"urich, 
8092 Z\"urich, Switzerland / {\tt lang@math.ethz.ch}}}
\date{July 29, 2011; revised, June 28, 2012}
\begin{document}

 
\maketitle

\begin{abstract}
Injective metric spaces, or absolute 1-Lipschitz retracts, share a number 
of properties with $\CAT(0)$ spaces. In the 1960es, J.~R.~Isbell showed that 
every metric space~$X$ has an injective hull~$\E{X}$. 
Here it is proved that if~$X$ is the vertex set of a connected locally 
finite graph with a uniform stability property of intervals, then~$\E{X}$ 
is a locally finite polyhedral complex with finitely many isometry types 
of $n$-cells, isometric to polytopes in $l^n_\infty$, for each~$n$.
This applies to a class of finitely generated groups~$\Gam$, including
all word hyperbolic groups and abelian groups, among others. 
Then~$\Gam$ acts properly on~$\E{\Gam}$ by cellular isometries, 
and the first barycentric subdivision of~$\E{\Gam}$ is a model 
for the classifying space~$\underbar{\rm E}\Gam$ for proper actions. 
If~$\Gam$ is hyperbolic, $\E{\Gam}$ is finite dimensional 
and the action is cocompact. In particular, every hyperbolic group 
acts properly and cocompactly on a space of non-positive
curvature in a weak (but non-coarse) sense. 
\end{abstract}
 

\section{Introduction}

A metric space $Y$ is called {\em injective} if for every metric space~$B$ 
and every $1$-Lipschitz map $f \colon A \to Y$ defined on a set $A \sub B$ 
there exists a $1$-Lipschitz extension $\ol f \colon B \to Y$ of $f$. 
The terminology is in accordance with the notion of an injective object 
in category theory. 
Basic examples of injective metric spaces are the real line, 
all complete $\R$-trees, and $l_\infty(I)$ for an arbitrary index set $I$. 
Every injective metric space $Y$ is complete, geodesic, and satisfies 
Busemann's non-positive curvature condition in a restricted form 
(see~\eqref{eq:gamxy} below); in particular, $Y$ is contractible.
By an old construction of Isbell~\cite{Isb},
every metric space $X$ possesses an essentially unique 
{\em injective hull} $(\e,\E{X})$; 
that is, $\E{X}$ is an injective metric space, $\e \colon X \to \E{X}$ 
is an isometric embedding, and every isometric embedding of $X$ into some 
injective metric space $Z$ factors through $\e$. 
If $X$ is compact then so is $\E{X}$, and if $X$ is finite then 
the injective hull is a finite polyhedral complex of dimension at 
most $\frac12 |X|$ whose $n$-cells are isometric to polytopes 
in~$l^n_\infty = l_\infty(\{1,\dots,n\})$.
A detailed account of injective metric spaces and hulls is given 
below, in Sections~\ref{Sect:inj} and~\ref{Sect:hull}.

Isbell's construction was rediscovered twenty years later by Dress~\cite{Dre}
(and even another time in~\cite{ChrL}).
Due to this independent work and a characterization of injective 
metric spaces from~\cite{AroP}, metric injective hulls 
are also called {\em tight spans}
or {\em hyperconvex hulls} in the literature,
furthermore ``hull'' is often substituted by ``envelope''. 
Tight spans are widely known in discrete mathematics 
and have notably been used in phylogenetic analysis 
(see~\cite{DreMT,DreHM} for some surveys). Apart from the two foundational 
papers~\cite{Isb,Dre} and some work referring to Banach spaces 
(see, for instance,~\cite{Isb2,Rao,CiaD}), the vast literature on metric 
injective hulls deals almost exclusively with finite metric spaces. 
Dress proved that for certain discrete metric spaces $X$ the tight span 
$T_X$ still has a polyhedral structure~\cite[(5.19), (6.2), (6.6)]{Dre}; 
these results, however, presuppose that $T_X$ is locally finite dimensional.
A simple sufficient, geometric condition on $X$ to this effect has been
missing (but see~\cite[Theorem~9 and~(5.14)]{Dre}). 

Here it is now shown that, in the case of integer valued metrics,  
a weak form of the fellow traveler property for discrete geodesics 
serves the purpose and even ensures that $\E{X}$ is proper, provided $X$ is;
see Theorem~\ref{Thm:intro-ex} below. 
The polyhedral structure of $\E{X}$ and the possible isometry types of cells 
are described in detail and no prior knowledge of the constructions 
in~\cite{Isb, Dre} is assumed.
With regard to applications in geometric group theory,
a general fixed point theorem for injective metric spaces is pointed out 
(Proposition~\ref{Prop:intro-fix}), which closely parallels the well-known 
result for $\CAT(0)$ spaces.
Furthermore, it has been known for some time that if the metric space $X$ is 
$\del$-hyperbolic, then so is $\E{X}$, and this implies that $\E{X}$ is within 
finite distance of $\e(X)$, provided $X$ is geodesic or discretely geodesic 
(Proposition~\ref{Prop:intro-hyp}). Despite this fact, the injective
hull of the hyperbolic plane has infinite topological dimension. Yet, it is 
shown that for a word hyperbolic group~$\Gam$ the injective hull 
is a finite dimensional polyhedral complex, on which $\Gam$ acts properly 
and cocompactly. This is part of a more general result, 
Theorem~\ref{Thm:intro-groups}, which provides a new source for geometric 
models of finitely generated groups and universal spaces for proper actions. 

To state these results in detail, we introduce some 
general notation used throughout the paper. Let $X$ be a
metric space with metric $d$. For $x,y \in X$,
\[
\bw(x,y) := \{ v \in X : d(x,v) + d(v,y) = d(x,y) \}
\]
denotes the {\em interval} between $x$ and $y$ (compare~\cite{Mul}), 
and for $x,v \in X$,
\begin{equation} \label{eq:c}
\co(x,v) := \{ y \in X : v \in \bw(x,y) \}
\end{equation}
is the {\em cone} determined by the directed pair $(x,v)$. 
Given a reference point $z \in X$, $d_z \colon X \to \R$ denotes the 
distance function to $z$, thus $d_z(x) = d(x,z)$.
The metric space $X$ is called {\em discretely geodesic} if the metric
is integer valued and for every pair of points $x,y \in X$ there 
exists an isometric embedding $\gam \colon \{0,1,\dots,d(x,y)\} \to X$ 
such that $\gam(0) = x$ and $\gam(d(x,y)) = y$. 
We say that a discretely geodesic metric space $X$ has 
{\em $\beta$-stable intervals}, for some constant $\beta \ge 0$, 
if for every triple of points $x,y,y' \in X$ with $d(y,y') = 1$ we have
\begin{equation} \label{eq:stable}
d_\H(\bw(x,y),\bw(x,y')) \le \beta,
\end{equation} 
where $d_\H$ denotes the Hausdorff distance in $X$. To verify this condition
it suffices, by symmetry, to show that for every $v \in \bw(x,y)$ there exists
a $v' \in \bw(x,y')$ with $d(v,v') \le \beta$; this means that
some (but not necessarily every) discrete geodesic from~$x$ to~$y'$ 
passes close to $v$.
We have the following result. 

\begin{Thm} \label{Thm:intro-ex}
Let $X$ be a discretely geodesic metric space such that all bounded 
subsets are finite. If $X$ has $\beta$-stable intervals, then the 
injective hull $\E{X}$ is proper (that is, bounded closed subsets are 
compact) and has the structure of a locally finite polyhedral 
complex with only finitely many isometry types of $n$-cells, isometric 
to injective polytopes in $l^n_\infty$, for every $n \ge 1$.
\end{Thm}

In this article, polytopes are understood to be convex and compact.
The polyhedral (or rather polytopal) structure 
of $\E{X}$ is discussed in detail, under some weaker but technical 
assumption, in Section~\ref{Sect:poly}. 
Then, in Section~\ref{Sect:cones}, this condition is approached through
the uniform stability of intervals in $X$. The proof of the theorem is
completed in Section~\ref{Sect:proofs}, where also the other results 
stated in this introduction are proved.
Inequality~\eqref{eq:stable} is used through the following two 
consequences. First, for every fixed vertex $v \in X$ there are only 
finitely many distinct cones $\co(x,v)$ as $x$ ranges over $X$; 
the argument goes back to Cannon~\cite{Can}.
Second, for all $x,y,z \in X$ there exists 
$v \in \bw(x,y)$ such that $d_z(v) \le \beta(d_z(x) + d_z(y) - d(x,y))$ 
(this implies in turn that $X$ has $2\beta$-stable intervals).
We derive upper bounds on the local dimension and 
complexity of $\E{X}$ in terms of the distance to a point $\e(z)$ in the 
image of the embedding $\e \colon X \to \E{X}$,
the cardinality of balls centered at $z$, and the constant $\beta$.
In particular, if there is a uniform bound on the number of points at 
distance one from any point in $X$, then every subcomplex of $\E{X}$ 
contained in a tubular neighborhood of $\e(X)$ is finite dimensional.
Note that this applies to finitely generated groups, discussed further below.

Injective (or hyperconvex) metric spaces have some remarkable 
fixed point properties. 
For instance, every $1$-Lipschitz map $L \colon X \to X$ 
of a bounded injective metric space $X$ has a non-empty fixed point set 
which is itself injective and thus contractible; 
compare~\cite[Theorem~6.1]{EspK} and the references there.
Contrary to what one might expect, the boundedness condition cannot be 
relaxed to the assumption that (the semigroup generated by) the $1$-Lipschitz 
map $L$ has bounded orbits. Indeed Prus gave an example
of an isometric embedding $L$ of the Banach space $l_\infty$ into itself 
such that $L$ has bounded orbits but no fixed point; 
see~\cite[Remark~6.3]{EspK}. However, this map $L$ is not surjective and 
thus the example still leaves room for the following proposition. 
In the process of finishing this paper I became aware of the 
reference~\cite{Dre3}, where the result is shown for continuous isometric 
actions of compact groups, using the Haar integral.

\begin{Prop} \label{Prop:intro-fix}
Let $X$ be an injective metric space. If $\Lam$ is a subgroup
of the isometry group of $X$ with bounded orbits,
then the fixed point set of $\Lam$ is non-empty and furthermore 
injective, hence contractible.
\end{Prop}

This should be compared with the analogous result for 
$\CAT(0)$ spaces (see~\cite[Corollary~II.2.8]{BriH}).

A metric space $X$ is called {\em $\del$-hyperbolic}, 
for some constant $\del \ge 0$, if 
\begin{equation} \label{eq:hyp}
d(w,x) + d(y,z) \le \max\{d(w,y) + d(x,z), d(x,y) + d(w,z)\} + \del
\end{equation}
for all quadruples of points $w,x,y,z \in X$. 
It is easily seen that every discretely
geodesic $\del$-hyperbolic metric space has $(\del+1)$-stable intervals.
The following proposition provides a most efficient way to embed a general 
$\del$-hyperbolic metric space isometrically into a geodesic 
(and contractible) $\del$-hyperbolic metric space
(see~\cite[Proposition~6.4.D]{Gro}, \cite[Theorem~4.1]{BonS} for some
results of similar type).

\begin{Prop} \label{Prop:intro-hyp}
Let $X$ be a $\del$-hyperbolic metric space. Then $\E{X}$ is 
$\del$-hyperbolic as well. If, in addition, $X$ is geodesic or discretely
geodesic, then $\E{X}$ is within distance $\del$ or $\del + \frac12$, 
respectively, of the image of the embedding $\e \colon X \to \E{X}$.
\end{Prop}

The first part of this result is mentioned without proof 
in~\cite[Section~4.4]{DreMT}, and the argument is given 
in~\cite[(4.1)]{Dre} for the case $\del = 0$.
The second part (with a primarily different proof and a worse bound)
served as the starting point for the present investigation and for the 
thesis~\cite{Moe}, where also a weaker version of 
Theorem~\ref{Thm:intro-groups} below was shown. 
 
Now let $\Gam$ be a group with a finite generating system $S$, equipped with
the word metric $d_S$ with respect to the alphabet $S \cup S^{-1}$. 
The isometric action of $\Gam$ by left multiplication on $\Gam_S = (\Gam,d_S)$ 
extends canonically to an isometric action on the injective hull $\E{\Gam_S}$. 
If $\Gam_S$ has $\beta$-stable intervals 
(see Remark~\ref{Rem:fft-property} for some sufficient conditions), 
Theorem~\ref{Thm:intro-ex} yields that $\E{\Gam_S}$ is a locally 
finite polyhedral complex with finitely many isometry types of $n$-cells 
for every $n$. By virtue of the two propositions above, we obtain the 
following further information.

\begin{Thm} \label{Thm:intro-groups}
Let $\Gam_S = (\Gam,d_S)$ be a finitely generated group.
If $\Gam_S$ has $\beta$-stable intervals, then $\Gam$ acts properly by 
cellular isometries on the complex $\E{\Gam_S}$, and the first barycentric 
subdivision $\E{\Gam_S}^1$ of $\E{\Gam_S}$ is a model for the classifying 
space $\underbar{\rm E}\Gam$ for proper actions. 
If $\Gam_S$ is $\del$-hyperbolic, then $\E{\Gam_S}$ is finite dimensional
and the action is cocompact in addition.
\end{Thm}

If $\Gam_S$ is $\del$-hyperbolic, $\E{\Gam_S}^1$ has only finitely many 
distinct $\Gam$-orbits of cells and thus constitutes a (so-called) finite 
model for $\underbar{\rm E}\Gam$. 
A corresponding result holds for the Rips complex $\cP_D(\Gam_S)$,
provided the maximal simplex diameter $D$ is chosen sufficiently large 
(see~\cite{MeiS}). It should be noted, by contrast, that the injective hull 
construction requires no such extra parameter and that the entire structure 
of $\E{\Gam_S}^1$ is canonically determined once the finite generating system 
$S$ is fixed. In some simple examples, the dimension of $\E{\Gam_S}$ is 
approximately one half the dimension of the smallest contractible Rips 
complex. It should further be noted that $\E{\Gam_S}$ comes with some 
features of non-positive curvature. 
In fact, for every injective metric space $X$, there is a map 
$\gam \colon X \times X \times [0,1] \to X$ such that 
$\gam_{xy} := \gam(x,y,\cdot)$ is a constant speed geodesic 
from $x$ to $y$ and 
\begin{equation} \label{eq:gamxy}
d(\gam_{xy}(t),\gam_{x'y'}(t)) \le (1-t)d(x,x') + td(y,y'),
\end{equation}
for all $x,y,x',y' \in X$ and $t \in [0,1]$. Thus $X$ satisfies Busemann's 
convexity condition for a suitable geodesic bicombing. In addition, 
$\gam$ can be chosen to be equivariant with respect to the full isometry 
group of $X$ (see~Proposition~\ref{Prop:bicombing}). 
Furthermore, as a consequence of~\eqref{eq:gamxy} 
and a result from~\cite{Wen}, injective metric spaces
satisfy isoperimetric filling inequalities of Euclidean type 
for integral cycles in any dimension, like $\CAT(0)$ spaces. 

These features relate Theorem~\ref{Thm:intro-groups} to the 
long-standing question in geometric group theory whether every word 
hyperbolic group acts properly and cocompactly by isometries on a 
$\CAT(0)$ or even $\CAT(-1)$ space. In the first instance, the result 
shows that the answer is positive if the $\CAT(0)$ condition is relaxed 
to the weaker convexity property of~\eqref{eq:gamxy}. The metric of 
$\E{\Gam_S}$ is piecewise of the type of $l^n_\infty$ and is therefore 
not $\CAT(0)$, unless $\E{\Gam_S}$ is one-dimensional (a tree). It is 
natural to ask whether $\E{\Gam_S}$ can be equipped with an honest 
equivariant $\CAT(0)$ metric. The answer turns out to be 
positive, for instance, if $\E{\Gam_S}$ has dimension two.
This and further results on the structure of injective
hulls of finitely generated groups will be discussed in a subsequent 
article. 


\section{Injective metric spaces} \label{Sect:inj}

We start by discussing some basic examples, properties, and 
characterizations of injective metric spaces. This section is largely 
expository. 

The set of all $1$-Lipschitz maps from a metric space
$B$ into another metric space $X$ will be denoted by $\Lip_1(B,X)$. 
Recall that $X$ is \emph{injective} if for every metric space $B$,
every $A \sub B$, and every $f \in \Lip_1(A,X)$ there exists 
$\ol f \in \Lip_1(B,X)$ such that $\ol f|_A = f$.
(Note that for $A = \es \ne B$ this says that $X \ne \es$.)

The most basic examples of injective metric spaces are the real line $\R$
and all non-empty closed subintervals, with the usual metric. 
For instance, if $f \in \Lip_1(A,\R)$, where $A \ne \es$ is a subset 
of a metric space $B$, then  
\begin{equation} \label{eq:least-ext}
\ol f(b) := \sup_{a \in A}(f(a) - d(a,b))
\end{equation}
defines the least possible extension $\ol f \in \Lip_1(B,\R)$ of $f$.

It follows easily from the definition that every injective metric space 
$X$ is complete and geodesic. Indeed, if $\bar X$ denotes the completion, 
then the identity map on $X$ extends to a $1$-Lipschitz retraction 
$\pi \colon \bar X \to X$ which turns out to be an isometry 
as $X$ is dense in $\bar X$. Furthermore, given $x,y \in X$, the map
that sends $0 \in \R$ to $x$ and $l := d(x,y)$ to $y$ extends to a 
$1$-Lipschitz map $\gam \colon [0,l] \to X$ which, due to the triangle 
inequality, is in fact an isometric embedding.

Another basic property is that for every triple
of points $x,y,z$ in an injective metric space $X$ there is a 
(not necessarily unique) median point $v \in X$, that is, a point in 
$\bw(x,y) \cap \bw(y,z) \cap \bw(z,x)$.
This is shown by extending the isometric inclusion $\{x,y,z\} \to X$
to a $1$-Lipschitz map from $Q := (\{x,y,z,u\},\bar d)$ to $X$, where
the metric $\bar d$ is determined by the requirement that it agrees 
with~$d$ on $\{x,y,z\}$ and that the additional point $u$ is a median point 
of $x,y,z$ in $Q$ (thus $\bar d(u,z) = (x \mid y)_z$, and so on,
see~\eqref{eq:Gromov-prd}). As above, this $1$-Lipschitz extension
is in fact an isometric embedding, and the image of $u$ is the desired 
median point $v \in Y$. 
One may choose geodesic segments $[v,x],[v,y],[v,z]$ to produce a 
geodesic tripod spanned by $x,y,z$ (thus $[v,x] \cup [v,y]$ is a geodesic 
segment from $x$ to $y$, and so on). 
Checking the existence of median points is a simple and useful first test
for injectivity. 
Furthermore, it follows that if $X$ is an injective metric space with 
the property that every pair of points $x,y$ is connected by a unique 
geodesic segment $[x,y]$, then every geodesic triangle in $X$ is a tripod
and so $X$ is an $\R$-tree. 
The converse is a well-known fact:

\begin{Prop} \label{Prop:tree-inj}
Every complete $\R$-tree $X$ is injective.
\end{Prop}

Most proofs in the literature proceed via pointwise extensions and transfinite
induction (compare Proposition~\ref{Prop:hyperconvex} below).
The following direct argument, extracted from a more 
general construction in~\cite{Lan}, adapts~\eqref{eq:least-ext} to trees.  

\begin{proof}
Fix a base point $z \in X$.
Let $f \in \Lip_1(A,X)$, where $\es \ne A \sub B$.
For every pair $(a,b) \in A \times B$, define 
\[
\rho(a,b) := \max\{ 0, d_z(f(a)) - d(a,b) \}
\]
and let $x(a,b)$ be the point on the segment $[z,f(a)]$ at 
distance $\rho(a,b)$ from $z$. For two such pairs $(a,b),(a',b')$, 
consider the tripod spanned by $z,f(a),f(a')$. 
Depending on the positions of $x(a,b)$ and $x(a',b')$ on the tripod,
$d(x(a,b),x(a',b'))$ equals either $|\rho(a,b) - \rho(a',b')|$ or 
$d(f(a),f(a')) - d(a,b) - d(a',b')$.
Since $d(f(a),f(a')) \le d(a,a')$, it follows that
\begin{equation} \label{eq:dvv}
d(x(a,b),x(a',b')) \le \max\{ |\rho(a,b) - \rho(a',b')|, d(b,b') \}.
\end{equation}
To define $\ol f \colon B \to X$ at the point $b \in B$, 
choose a sequence $(a_i)$ in $A$ such that 
\[
\lim_{i \to \infty} \rho(a_i,b) = \bar\rho(b) := \sup_{a \in A} \rho(a,b).
\]
The corresponding sequence $(x(a_i,b))$ in $X$ is Cauchy by~\eqref{eq:dvv},
and $\ol f(b)$ is defined as its limit, which is independent
of the choice of $(a_i)$. Note that $\bar\rho \colon B \to \R$ 
is the least non-negative $1$-Lipschitz extension of $d_z(f(\cdot))$ 
(compare~\eqref{eq:least-ext}).
It follows from~\eqref{eq:dvv} that $\ol f \in \Lip_1(B,X)$. 
To check that $\ol f$ extends $f$, let $b \in A$, and let $(a_i)$ be a 
sequence in $A$ such that $\rho(a_i,b) \to \bar\rho(b)$. 
We have $\rho(b,b) = d_z(f(b)) = \bar\rho(b)$ and $x(b,b) = f(b)$,
so $d(x(a_i,b),f(b)) \le |\rho(a_i,b) - \bar\rho(b)|$ by~\eqref{eq:dvv}.
Since $x(a_i,b) \to \ol f(b)$, this gives $\ol f(b) = f(b)$. 
\end{proof}

The $l_\infty$ product of a non-empty family $\{(X_i,d_i,z_i)\}_{i\in I}$
of pointed metric spaces is defined as the set 
of all $x = (x_i)_{i \in I}$ with $x_i \in X_i$ and 
$\sup_{i \in I} d_i(x_i,z_i) < \infty$,
endowed with the metric $(x,x') \mapsto \sup_{i\in I} d_i(x_i,x'_i)$.
Here $I \ne \es$ is an arbitrary index set; if $I$ is finite
or the diameters of the $X_i$ are uniformly bounded, base points may be 
disregarded. It is easy to see that if each $(X_i,d_i)$ is injective, 
then so is the $l_\infty$ product. 
In case $(X_i,z_i) = (\R, 0)$ for all $i \in I$, the corresponding 
$l_\infty$ product is the Banach space $l_\infty(I)$, which is thus an 
injective metric space. Similarly, $L_\infty(Y,\mu)$ is injective for 
every measure space $(Y,\mu)$.

Next we recall some well-known characterizations of injective
metric spaces. A metric space $X$ is called an 
\emph{absolute $1$-Lipschitz retract} 
if, whenever $i \colon X \to Y$ is an isometric embedding into
another metric space $Y$, there exists a $1$-Lipschitz retraction 
of $Y$ onto $i(X)$. If $X$ is injective and $i \colon X \to Y$ 
is an isometric embedding, then $i(X)$ is injective
and thus the identity map on $i(X)$ extends to 
a $1$-Lipschitz retraction $\pi \colon Y \to i(X)$.
On the other hand, every metric space $X$ embeds isometrically 
into $l_\infty(X)$ via the map $k_z \colon x \mapsto d_x - d_z$, 
for any base point $z \in X$. Hence, if $k_z(X)$ is a $1$-Lipschitz retract 
in $l_\infty(X)$, then $X$ is injective since $l_\infty(X)$ is.
This shows:

\begin{Prop} \label{Prop:alr}
A metric space $X$ is injective if and only if it is an absolute
$1$-Lipschitz retract.
\end{Prop}

At this point, we note that every injective metric space $X$
is contractible. If $\pi \colon l_\infty(X) \to X' := k_z(X)$ is a 
$1$-Lipschitz retraction onto the image of the Kuratowski embedding, then 
$h(x',t) := \pi(tx')$ defines a homotopy 
$h \colon X' \times [0,1] \to X'$ from the constant map with value
$0 = k_z(z)$ to the identity map. A map $\gam$ as in~\eqref{eq:gamxy} 
can be obtained in a similar way. See also 
Proposition~\ref{Prop:bicombing} and~\cite[Theorem~1.1]{Isb}.

Another characterization of injective metric spaces relies on pointwise 
extensions of $1$-Lipschitz maps. A metric space $X$ is said to be 
\emph{hyperconvex} if every family $((x_i,r_i))_{i \in I}$ in $X \times \R$ 
with the property that $r_i + r_j \ge d(x_i,x_j)$ for all pairs of 
indices $i,j \in I$ satisfies $\bigcap_{i \in I}B(x_i,r_i) \ne \es$.
(We adopt the convention that the intersection equals $X$ if $I = \es$,
so that hyperconvex spaces are non-empty by definition.)
This terminology was introduced
by Aronszajn and Panitchpakdi in~\cite{AroP}, who also observed 
Proposition~\ref{Prop:alr} and the next result. 

\begin{Prop} \label{Prop:hyperconvex}
A metric space $X$ is injective if and only if it is hyperconvex.
\end{Prop}

For the proof, one notes first that if $f \in \Lip_1(A,X)$, $\es \ne A \sub B$, 
and $b \in B \sm A$, then $d(a,b) + d(a',b) \ge d(a,a') \ge d(f(a),f(a'))$ 
for all $a,a' \in A$. Hence, if $X$ is hyperconvex, then 
$\bigcap_{a \in A}B(f(a),d(a,b))$ is non-empty, 
and one obtains an extension $f_b \in \Lip_1(A \cup \{b\},X)$ of $f$ 
by declaring $f_b(b)$ to be any point in this intersection.
By iterating this process transfinitely, in general, one 
infers that $X$ is injective. Conversely, if $X$ is injective,
a similar argument as for the existence of median points shows that
$X$ is hyperconvex. A useful direct consequence of this characterization 
is that the intersection of a family of closed balls in an injective 
metric space is injective, whenever the intersection is non-empty.
Some key results on hyperconvex metric spaces were shown by Baillon~\cite{Bai}. 
Proposition~\ref{Prop:hyperconvex} will only be used in 
Section~\ref{Sect:proofs} for the proof of Proposition~\ref{Prop:intro-fix}.

A concept very close to hyperconvexity is the 
{\em binary intersection property}, obtained by replacing 
the inequality $r_i + r_j \ge d(x_i,x_j)$ in the above definition by 
the condition $B(x_i,r_i) \cap B(x_j,r_j) \ne \es$. The two concepts
agree for geodesic metric spaces. Nachbin~\cite[Theorem~1]{Nac} showed 
that a normed real vector space $X$ has the binary intersection property if 
and only if $X$ is {\em linearly injective}, that is, for every real
normed space $B$, every linear subspace $A \sub B$, 
and every bounded linear operator $f \colon A \to X$ there exists a 
linear extension $\ol f \colon B \to X$ with norm $\|\ol f\| = \|f\|$. 
(The Hahn--Banach Theorem thus asserts that $\R$ is linearly injective.) 
Hence, a real normed space $X$ is injective as a metric space if and only 
if $X$ is injective in the linear category, and no ambiguity arises. 
By~\cite[Theorem~3]{Nac}, an $n$-dimensional normed space $X$ is injective 
if and only if $X$ is linearly isometric to 
$l_\infty^n$ or, in other words, balls in $X$ are parallelotopes. 
The final classification result, usually attributed to 
Nachbin--Goodner--Kelley, asserts that a real normed space is injective 
if and only if it is isometrically isomorphic to the Banach space $C(K)$ of 
continuous real valued functions on some extremally disconnected compact 
Hausdorff space $K$, endowed with the supremum norm. See~\cite{Kel}.

It is clear that linear subspaces of injective normed spaces need 
not be injective. A familiar example is the plane 
$H = \{x_1 + x_2 + x_3 = 0\}$ in $l_\infty^3$, whose norm ball is hexagonal. 
One may also check directly that the triple of points 
$(1,1,-2),(1,-2,1),(-2,1,1)$ has no median point in $H$. 
We conclude this section by showing that certain subsets of 
$l_\infty^n$ (or $l_\infty(I)$) defined by linear inequalities involving 
at most two variables are injective. This will be employed
to prove that the polyhedral cells of $\E{X}$ are themselves injective 
(compare Theorem~\ref{Thm:intro-ex}); however, this fact will not be used
further in this paper.

\begin{Prop} \label{Prop:p-inj}
Let $I \ne \es$ be any index set. Suppose that $Q$ is a non-empty subset
of $l_\infty(I)$ given by an arbitrary system of inequalities
of the form $\sig x_i \le C$ or $\sig x_i + \tau x_j \le C$
with $|\sig|,|\tau| = 1$ and $C \in \R$. Then $Q$ is injective.
\end{Prop}

We use a similar explicit construction as for $\R$-trees.
Some further results on injective polyhedral sets in $l_\infty^n$ can be found 
in~\cite[Section~1.8.2]{Moe}. 
A good characterization of such sets seems to be missing.

\begin{proof}
Assume that $0 \in Q$, so that all constants on the right sides
of the inequalities describing $Q$ are non-negative. 
For $i \in I$, denote by $R_i$ the reflection of $l_\infty(I)$ that 
interchanges $x_i$ with $-x_i$.
Let $B$ be a metric space and $\es \ne A \sub B$.
We show that there exists an extension operator 
$\phi \colon \Lip_1(A,l_\infty(I)) \to \Lip_1(B,l_\infty(I))$ such that 
\begin{equation} \label{eq:l1}
\phi(R_i \circ f) = R_i \circ \phi(f)
\end{equation}
for every $i$, and such that the components of $\phi(f)$ satisfy
\begin{equation} \label{eq:l2}
\phi(f)_i + \phi(f)_j \le C
\end{equation}
whenever $f_i + f_j \le C$ for some pair of possibly equal indices $i,j$
and some constant $C \ge 0$. This clearly gives the result. 

First, for a real valued function 
$f \in \Lip_1(A,\R)$, we combine the smallest and largest $1$-Lipschitz
extensions and define $\ol f \colon B \to \R$ by
\[
\ol f(b) := \sup \Bigl\{ 0,\,\sup_{a \in A}(f(a) - d(a,b)) \Bigr\} + 
\inf \Bigl\{ 0,\,\inf_{a' \in A}(f(a') + d(a',b)) \Bigr\}.
\]
Note that at most one of the two summands is nonzero since 
$f(a) - d(a,b) \le f(a') + d(a,a') - d(a,b) \le f(a') + d(a',b)$.
It is not difficult to check that $\ol f$ is a $1$-Lipschitz extension 
of $f$ and that $\ol{R \circ f} = R \circ \ol f$ for the reflection 
$R \colon x \mapsto -x$ of $\R$. 
(The proof of Proposition~\ref{Prop:tree-inj} yields precisely 
this extension $\ol f$ in the case $(X,z) = (\R,0)$.)

Now, for $f \in \Lip_1(A,l_\infty(I))$,
define $\phi(f)$ such that $\phi(f)_i = \ol{f_i}$ for every~$i$.
Clearly $\phi(f) \in \Lip_1(B,l_\infty(I))$, and~\eqref{eq:l1} holds. 
As for~\eqref{eq:l2}, suppose that $f_i + f_j \le C$ for some indices $i,j$
and some constant $C \ge 0$. Let $b \in B$, and assume that  
$\phi(f)_i(b) \ge \phi(f)_j(b)$. If $\phi(f)_j(b) > 0$, then
\begin{align*}
\phi(f)_i(b) + \phi(f)_j(b)
&= \sup_{a,a' \in A} (f_i(a) + f_j(a') - d(a,b) - d(a',b)) \\
&\le \sup_{a,a' \in A} (f_i(a) + f_j(a') - d(a,a')) \\
&\le \sup_{a \in A} (f_i(a) + f_j(a)) \le C.
\end{align*}
If $\phi(f)_i(b) > 0 \ge \phi(f)_j(b)$, then 
\begin{align*}
\phi(f)_i(b) + \phi(f)_j(b)
&\le \sup_{a \in A} (f_i(a) - d(a,b)) + \inf_{a' \in A} (f_j(a') + d(a',b)) \\
&\le \sup_{a \in A} (f_i(a) + f_j(a)) \le C.
\end{align*}
Finally, if $\phi(f)_i(b) \le 0$, then 
$\phi(f)_i(b) + \phi(f)_j(b) \le 0 \le C$.
\end{proof}


\section{Injective hulls} \label{Sect:hull}

We now review Isbell's~\cite{Isb} construction $X \mapsto \E{X}$ in detail.
Our proof of the injectivity of $\E{X}$ differs from Isbell's in that
it does not appeal to Zorn's Lemma or the like. Instead we employ an 
observation by Dress, restated in Proposition~\ref{Prop:retr} below, 
which will be of further use.

Let $X$ be a (non-empty) metric space. Denote by $\R^X$ the 
vector space of all real valued functions on $X$,
and define
\[
\D{X} := \{ f \in \R^X : 
\text{$f(x) + f(y) \ge d(x,y)$ for all $x,y \in X$}\} 
\]
(compare~\cite[p.~35]{Nac}). By the triangle inequality, the distance 
function $d_z$ belongs to $\D{X}$ for each $z \in X$, and clearly
all elements of $\D{X}$ are non-negative.
Isbell called a function $f \in \R^X$ \emph{extremal} 
if it is a minimal element of the partially ordered set $(\D{X},\le)$, 
where $g \le f$ means $g(x) \le f(x)$ for all $x \in X$ as usual. Thus
\[
\E{X} := \{ f \in \D{X} : 
\text{if $g \in \D{X}$ and $g \le f$, then $g = f$}\} 
\]
is the set of extremal functions on $X$. 
In case $X$ is compact, $f \in \D{X}$ is extremal if and only if 
for every $x \in X$ there exists $y \in X$ such that 
$f(x) + f(y) = d(x,y)$. In general, $f \in \R^X$ is extremal if 
and only if 
\begin{equation} \label{eq:extr-sup}
f(x) = \sup_{y \in X}(d(x,y) - f(y))
\end{equation}
for all $x \in X$. Each $d_z$ is extremal.
Applying~\eqref{eq:extr-sup} twice one obtains
\[
f(x) \le \sup_{y \in X}(d(x,x') + d(x',y) - f(y)) = d(x,x') + f(x')
\]
for all $x,x' \in X$, so every $f \in \E{X}$ is $1$-Lipschitz.
Now consider the set
\[
\Dl{X} := \D{X} \cap \Lip_1(X,\R),
\]
equipped with the metric
\[
(f,g) \mapsto \|f - g\|_\infty = \sup_{x \in X} |f(x) - g(x)|.
\]
To see that the supremum is finite, note that 
a function $f \in \R^X$ belongs to $\Dl{X}$ if and only if 
$|f(x) - d(x,y)| \le f(y)$ for all $x,y \in X$ or, equivalently, 
\begin{equation} \label{eq:fdf}
\|f - d_y\|_\infty = f(y) 
\end{equation}
for all $y \in X$. Hence, $\|f - g\|_\infty \le \inf(f+g)$.
The set $\E{X}$ is contained in $\Dl{X}$ and 
is equipped with the induced metric. The map
\[
\e \colon X \to \E{X}, \quad \e(y) = d_y,
\]
is a canonical isometric embedding of $X$ into $\E{X}$, as
$\|d_y - d_z\|_\infty = d(y,z)$. Equation~\eqref{eq:fdf} will be used 
frequently. It shows that a function $f \in \Dl{X}$ corresponds, after 
identification of $X$ with $\e(X)$, to the restriction of a distance 
function to a point in $\Dl{X}$, namely $f$ itself. 

To prove that $(\e,\E{X})$ is an injective hull of $X$
we shall make use of the following important fact.

\begin{Prop} \label{Prop:retr}
For every metric space $X$ there exists a map 
$p \colon \D{X} \to \E{X}$ such that 
\begin{enumerate}
\item[\rm (1)]
$p(f) \le f$ for all $f \in \D{X}$, hence $p(f) = f$ for all $f \in \E{X}$;
\item[\rm (2)]
$\|p(f) - p(g)\|_\infty \le \|f - g\|_\infty$ for all $f,g \in \D{X}$.
\end{enumerate}
\end{Prop}

In~(2) the right side is possibly infinite, but it is finite
if $f,g \in \Dl{X}$, thus the restriction of $p$ to $\Dl{X}$ is a 
$1$-Lipschitz retraction onto $\E{X}$. The existence of such a map $p$
could be shown by means of Zorn's Lemma, however 
Dress~\cite[Section~(1.9)]{Dre} (compare~\cite[Lemma~5.3]{ChrL}) 
also found the following construction, 
which is canonical in the sense that no choices need to be made.

\begin{proof}
For every $f \in \D{X}$, define $f^* \in \R^X$ such that 
\[
f^*(x) = \sup_{z \in X}(d(x,z) - f(z))
\] 
for all $x \in X$. Clearly $f^* \le f$, and equality holds
if and only if $f \in \E{X}$, by~\eqref{eq:extr-sup}. 
For every pair of points $x,y \in X$, the definition of $f^*$ gives 
$f^*(x) + f(y) \ge d(x,y)$ and $f(x) + f^*(y) \ge d(x,y)$.
It follows that the function
\[
q(f) := \frac12(f+f^*) 
\]
belongs to $\D{X}$, and $q(f) \le f$. For all $f,g \in \D{X}$ and $x \in X$,
\[
g^*(x) = \sup_{z \in X}(d(x,z) - f(z) + f(z) - g(z)) 
\le f^*(x) + \|f - g\|_\infty,
\] 
hence $\|f^* - g^*\|_\infty \le \|f - g\|_\infty$ and thus
\[
\|q(f) - q(g)\|_\infty 
\le \frac12 \|f - g\|_\infty + \frac12 \|f^* - g^*\|_\infty
\le \|f - g\|_\infty.
\]
Iterating the map $q$, we obtain for every $f \in \D{X}$ a sequence 
of functions $q(f) \ge q^2(f) \ge q^3(f) \ge \ldots$ 
in $\D{X}$, then we define $p(f)$ as the pointwise limit.
Clearly $p(f) \in \D{X}$, and~(1) and~(2) hold. For all $n \ge 1$,
$p(f) \le q^n(f)$ and hence $p(f)^* \ge q^n(f)^*$, so
\[
0 \le p(f) - p(f)^* \le q^n(f) - q^n(f)^* = 2(q^n(f) - q^{n+1}(f)).
\]
As $n \to \infty$, the last term converges pointwise to $0$,
thus $p(f)^* = p(f)$ and therefore $p(f) \in \E{X}$.
\end{proof}

Now, since $\E{X}$ is a $1$-Lipschitz retract of $\Dl{X}$, to prove
the injectivity of $\E{X}$ it remains to show that $\Dl{X}$ is injective.
A simple component-wise extension procedure applies, like for $l_\infty(I)$.

\begin{Prop} \label{Prop:dl-inj}
For every metric space $X$ the metric spaces $\Dl{X}$ and $\E{X}$ 
are injective.
\end{Prop}

\begin{proof}
As just mentioned, in view of Proposition~\ref{Prop:retr} it suffices
to prove the result for $\Dl{X}$. We could embed $\Dl{X}$
into $l_\infty(X)$ via $f \mapsto f-h$ for some fixed $h \in \Dl{X}$
and then refer to Proposition~\ref{Prop:p-inj}, 
but the following argument is slightly more direct.
Let $B$ be a metric space, $\es \ne A \sub B$, and let 
$F \colon A \to \Dl{X}$ be a $1$-Lipschitz map, $F \colon a \mapsto f_a$.
For $b \in B$, put
\[
\bar f_b(x) := \inf_{a \in A}(f_a(x) + d(a,b))
\]
for all $x \in X$. Clearly $\bar f_b$ is a non-negative $1$-Lipschitz 
function on $X$, as the infimum of a family of such.
For $a,a' \in A$ and $y \in X$, we have 
$f_a(y) - f_{a'}(y) \le \|f_a - f_{a'}\|_\infty 
= \|F(a) - F(a')\|_\infty \le d(a,a')$ and so
\begin{align*}
\bar f_b(x) + \bar f_b(y) 
&\ge \inf_{a,a' \in A} (f_a(x) + f_{a'}(y) + d(a,a')) \\
&\ge \inf_{a \in A}(f_a(x) + f_a(y)) \\
&\ge d(x,y). 
\end{align*}
This shows that $\bar f_b \in \Dl{X}$. For $b,b' \in B$ and $x \in X$,
\[
\bar f_b(x) - d(b,b') 
= \inf_{a \in A}(f_a(x) + d(a,b) - d(b,b'))
\le \bar f_{b'}(x),
\]
hence $\|\bar f_b - \bar f_{b'}\|_\infty \le d(b,b')$.
If $b \in A$, then $\bar f_b(x) \le f_b(x)$ 
and $f_b(x) \le f_a(x) + \|f_a - f_b\|_\infty \le f_a(x) + d(a,b)$ for 
all $x \in X$ and $a \in A$, so that $\bar f_b = f_b$.
Thus $\ol F \colon b \mapsto \bar f_b$ is a $1$-Lipschitz extension
of $F$.
\end{proof}

If $X$ is finite, so that the supremum norm gives a metric on $\D{X}$,
the same argument also shows that $\D{X}$ is injective. 

We now state Isbell's result about $\E{X}$. For brevity,
isometric embeddings will just be called {\em embeddings}.
An embedding $i$ of $X$ into some metric space $Y$ is called
{\em essential} if for every metric space $Z$ and every $1$-Lipschitz map 
$h \colon Y \to Z$ with the property that $h \circ i \colon X \to Z$ is an 
embedding, $h$ is an embedding as well. 
If $i \colon X \to Y$ is essential and $Y$ is injective, 
then $(i,Y)$ is an {\em injective hull} of $X$; 
see~\cite[Section~9]{AdaHS}. In the terminology of~\cite{Dre}, 
an essential extension $(i,Y)$ of $X$ is called a {\em tight extension}
(and $\D{X}$ and $\E{X}$ are denoted $P_X$ and $T_X$, respectively). 

\begin{Thm} \label{Thm:isbell}
For every metric space $X$, the following hold:
\begin{enumerate}
\item[\rm (1)] 
If $L \colon \E{X} \to \E{X}$ is a $1$-Lipschitz map that fixes $\e(X)$ 
pointwise, then $L$ is the identity on $\E{X}$;
\item[\rm (2)]
$(\e,\E{X})$ is an injective hull of $X$;
\item[\rm (3)]
if $(i,Y)$ is another injective hull of $X$, then there exists 
a unique isometry $I \colon \E{X} \to Y$ with the property that 
$I \circ \e = i$.
\end{enumerate}
\end{Thm}

\begin{proof}
For~(1) we use~\eqref{eq:fdf}. The map $L$ takes 
$f \in \E{X}$ to some $g \in \E{X}$ such that 
\[
g(x) = \|g - d_x\|_\infty 
= \|L(f) - L(d_x)\|_\infty \le \|f - d_x\|_\infty = f(x)
\]
for all $x \in X$, so $g = f$ by the minimality of $f$.

By Proposition~\ref{Prop:dl-inj}, $\E{X}$ is injective, so for~(2) 
it remains to show that $\e$ is essential. 
Suppose $h \colon \E{X} \to Z$ is $1$-Lipschitz and 
$h \circ \e \colon X \to Z$ is an embedding.
Since $\E{X}$ is injective, $\e \colon X \to \E{X}$ extends to a $1$-Lipschitz 
map $\ol\e \colon Z \to \E{X}$, thus $\ol\e \circ h \circ \e = \e$.
The map $L := \ol\e \circ h$ is $1$-Lipschitz and fixes $\e(X)$ pointwise, 
so $L$ is the identity on $\E{X}$ by~(1). 
As both $h$ and $\ol\e$ are $1$-Lipschitz, $h$ is in fact an 
embedding.  

As for~(3), if $(i,Y)$ is another injective hull of $X$, then $i$ extends to
a $1$-Lipschitz map $I \colon \E{X} \to Y$, so $I \circ \e = i$.
Likewise, there is a $1$-Lipschitz map $\ol\e \colon Y \to \E{X}$ with
$\ol\e \circ i = \e$. Since $i$ is essential, $\ol\e$ is an embedding; 
furthermore, $\ol\e \circ I \circ \e = \e$, thus $\ol\e \circ I = \id_{\E{X}}$ 
by~(1). Hence $\ol\e$ is an isometry onto $\E{X}$, and $I$ is its inverse.
\end{proof}

Injective hulls can be characterized in a number of 
different ways. We just state the following proposition, 
which is independent of the construction described above, except 
that the proof relies on the existence of {\em some} injective hull
of $X$. For the details we refer again to the general discussion
in~\cite[Section~9]{AdaHS}.

\begin{Prop}
Let $X$ and $Y$ be metric spaces, and let $i \colon X \to Y$ be an embedding.
Then the following are equivalent:
\begin{enumerate}
\item[\rm (1)]
$(i,Y)$ is an injective hull of $X$, that is, $i$ is essential and $Y$ 
is injective;
\item[\rm (2)]
$(i,Y)$ is a maximal essential extension of $X$, that is,
$i$ is essential and $Y$ has no proper essential extension;
\item[\rm (3)]
$(i,Y)$ is a minimal injective extension of $X$, that is,
$Y$ is injective and no proper subspace of $Y$ 
containing $i(X)$ is injective;
\item[\rm (4)]
$(i,Y)$ is a smallest injective extension of $X$, that is,
$Y$ is injective and whenever $j \colon X \to Z$ is an embedding 
into some injective metric space $Z$, there is an embedding 
$h \colon Y \to Z$ such that $h \circ i = j$.
\end{enumerate}
\end{Prop}

In fact, (3)~is the definition of injective hulls adopted 
by Isbell in~\cite{Isb}, and~(2) corresponds to the notion
of {\em tight span} introduced by Dress~\cite{Dre}.
In the introduction we used property~(4),
a concrete instance of which is given in the next result
(compare \cite[Section~(1.11)]{Dre}).

\begin{Prop} \label{Prop:x-subspace}
Let $X$ be a subspace of the metric space $X'$. Then:
\begin{enumerate}
\item[\rm (1)]
There exists an isometric embedding $h \colon \E{X} \to \E{X'}$ 
such that $h(f)|_X = f$ for every $f \in \E{X}$. 
\item[\rm (2)]
For every pair of functions $g \in \E{X}$ and $f' \in \E{X'}$ there 
exists $g' \in \E{X'}$ such that $g'|_X = g$ and 
$\|g' - f'\|_\infty = \|g - f'|_X\|_\infty$.
\end{enumerate}
\end{Prop}

\begin{proof}
For $f \in \E{X}$, let first $\ol f \colon X' \to \R$
be the $1$-Lipschitz extension defined by 
\[
\ol f(y) := \inf_{x \in X}(f(x) + d(x,y)).
\]
Clearly $\ol f \in \D{X'}$. Now put 
$h(f) := p(\ol f) \in \E{X'}$,
where $p$ is as in Proposition~\ref{Prop:retr}. We have $h(f)|_X 
= p(\ol f)|_X \le \ol f|_X = f$; since $h(f)|_X \in \D{X}$, 
this gives $h(f)|_X = f$ by the minimality of $f$.
For $f,g \in \E{X}$ and $y \in X'$,
\[
\ol f(y) - \|f - g\|_\infty 
= \inf_{x \in X}(f(x) - \|f - g\|_\infty + d(x,y))
\le \ol g(y),
\]
hence $\|h(f) - h(g)\|_\infty
= \|p(\ol f) - p(\ol g)\|_\infty
\le \|\ol f - \ol g\|_\infty = \|f - g\|_\infty$.
This yields~(1).

As for~(2), suppose that $\nu := \|g - f'|_X\|_\infty < \infty$.
Define $\tilde g \colon X' \to \R$ such that $\tilde g|_X = g$ and 
$\tilde g(y) = f'(y) + \nu$ for all $y \in X' \sm X$.
Since $\tilde g(x) = g(x) \ge f'(x) - \nu$ for $x \in X$, it follows that 
$\tilde g \in \D{X'}$. Now let $g' \in \E{X'}$ be any extremal function
with $g' \le \tilde g$. Similarly as above, $g'|_X \le \tilde g|_X = g$
and thus $g'|_X = g$ by the minimality of $g$. 
Furthermore, $g' \le f' + \nu$ and hence
\[
g'(y) \ge \sup_{y' \in X'}(d(y,y') - f'(y') - \nu) = f'(y) - \nu
\]
for all $y \in X'$ by~\eqref{eq:extr-sup}. This gives the result.
\end{proof}

A number of properties of $\E{X}$ are more or less obvious from 
the construction. If $X$ is bounded, then $0 \le f \le \diam(X) := 
\sup_{x,y \in X}d(x,y)$ for all $f \in \E{X}$ by~\eqref{eq:extr-sup}, thus
\[
\diam(\E{X}) \le \diam(X).
\]
If $X$ is compact, then so is $\E{X}$, as a consequence of
the Arzel\`a-Ascoli Theorem. If $X$ is finite, $\E{X}$ is a polyhedral 
subcomplex of the boundary of the polyhedral set $\D{X} \sub \R^X$.
The faces of $\D{X}$ that belong to $\E{X}$ are exactly those whose
affine hull $H \sub \R^X$ is determined by a system of equations 
of the form $f(x_i) + f(x_j) = d(x_i,x_j)$ involving each point 
$x_i \in X$ at least once. (Note that these are precisely the bounded faces 
of $\D{X}$; compare \cite[Lemma~1]{Dre2}.) 
It follows that $\E{X}$ has dimension at most $\frac12 |X|$.
The possible combinatorial types of the injective hulls of metric spaces
up to cardinality~$5$ are depicted in~\cite[Section~(1.16)]{Dre}, 
and a classification for $6$-point metrics is given in~\cite{StuY}.
 
\begin{Rem} \label{Rem:banach}
As mentioned in Section~\ref{Sect:inj}, a normed real vector 
space is injective as a metric space if and only if it is linearly injective,
and the only $n$-dimensional example is $l^n_\infty$, up to isometric 
isomorphism.
Cohen~\cite{Coh} showed that every real or complex normed space 
has an essentially unique injective hull in the respective
linear category. Isbell~\cite{Isb2} and Rao~\cite{Rao} then proved that 
for a real normed space~$X$ the linearly injective hull is isometric 
to~$\E{X}$; an explicit description of the Banach space structure on~$\E{X}$ 
can be found in~\cite{CiaD}.
\end{Rem}

We conclude this section with some results involving isometries of $X$.
The isometry group of $X$ will be denoted by $\Isom(X)$.

\begin{Prop} \label{Prop:isometries}
Let $X$ be a metric space. Then:
\begin{enumerate}
\item[\rm (1)] 
For every $L \in \Isom(X)$ there is a unique isometry 
$\bar{L} \colon \E{X} \to \E{X}$ 
with the property that $\bar{L} \circ \e = \e \circ L$.
One has $\bar L(f) = f \circ L^{-1}$ for all $f \in \E{X}$, 
and $(L,f) \mapsto \bar L(f)$ is an action of $\Isom(X)$ on $\E{X}$.
\item[\rm (2)]
For every $L \in \Isom(X)$, the linear isomorphism $f \mapsto f \circ L^{-1}$ 
of $\R^X$ maps $\D{X}$ onto itself, and the map $p \colon \D{X} \to \E{X}$ 
constructed in the proof of Proposition~\ref{Prop:retr} has the additional 
property that $\bar{L}(p(f)) = p(f \circ L^{-1})$ for all $f \in \D{X}$.
\end{enumerate}
\end{Prop}

\begin{proof}
For every $L \in \Isom(X)$, $\e \circ L$ is essential and so 
$(\e \circ L,\E{X})$ is an injective hull of $X$.
Hence, by part~(3) of Theorem~\ref{Thm:isbell}, there is a unique isometry
$\bar L \colon \E{X} \to \E{X}$ such that $\bar L \circ \e = \e \circ L$. 
If $f \in \E{X}$ and $x \in X$, then
\begin{align*}
(\bar{L}(f))(x) 
&= \|\bar{L}(f) - d_x\|_\infty
= \|f - \bar{L}^{-1}(d_x)\|_\infty 
= \|f - d_{L^{-1}(x)}\|_\infty\\ 
&= f(L^{-1}(x)).
\end{align*}
Obviously $(L,f) \mapsto \bar L(f) = f \circ L^{-1}$ 
is an action of $\Isom(X)$ on $\E{X}$.

It is straightforward to check that the linear isomorphism 
$f \mapsto f \circ L^{-1}$ of $\R^X$ 
maps $\D{X}$ onto $\D{X}$ and that it commutes with the operators defined 
in the proof of Proposition~\ref{Prop:retr}, thus
$f^* \circ L^{-1} = (f \circ L^{-1})^*$, 
$q(f) \circ L^{-1} = q(f \circ L^{-1})$, and 
\[
\bar L(p(f)) = p(f) \circ L^{-1} = p(f \circ L^{-1})
\]
(compare~\cite[pp.~83--84]{Dre3}). 
\end{proof}

As an application of the above result 
we show that the weak convexity property of injective 
metric spaces stated in~\eqref{eq:gamxy} holds in an equivariant form.
By a \emph{geodesic bicombing} $\gam$ on a metric space $X$ 
we mean a map $\gam \colon X \times X \times [0,1] \to X$ 
such that, for every pair $(x,y) \in X \times X$, 
$\gam_{xy} := \gam(x,y,\cdot)$ is a geodesic from $x$ to $y$ with 
constant speed, that is, $\gam_{xy}(0) = x$, $\gam_{xy}(1) = y$, and 
$d(\gam_{xy}(s),\gam_{xy}(t)) = (t - s)d(x,y)$ for $0 \le s \le t \le 1$.

\begin{Prop} \label{Prop:bicombing}
Every injective metric space $X$ admits a geodesic bicombing $\gam$ 
such that, for all $x,y,x',y' \in X$ and $t \in [0,1]$,
\begin{enumerate}
\item[\rm (1)] 
$d(\gam_{xy}(t),\gam_{x'y'}(t)) \le (1-t)d(x,x') + td(y,y')$;
\item[\rm (2)]
$\gam_{xy}(t) = \gam_{yx}(1-t)$;
\item[\rm (3)]
$L \circ \gam_{xy} = \gam_{L(x)L(y)}$ for every isometry $L$ of $X$.
\end{enumerate}
\end{Prop}

\begin{proof}
Since $X$ is injective, the canonical map $\e \colon x \mapsto d_x$
is an isometry of $X$ onto $\E{X}$. Let $p \colon \D{X} \to \E{X}$ be the map
from the proof of Proposition~\ref{Prop:retr}. 
For all $x,y \in X$ and $t \in [0,1]$, we have $(1-t)d_x + td_y \in \Dl{X}$, 
and we set
\[
\gam_{xy}(t) := (\e^{-1} \circ p)((1-t)d_x + td_y).
\]
Since $p|_{\Dl{X}}$ is $1$-Lipschitz, it follows that
\begin{align*}
d(\gam_{xy}(t),\gam_{x'y'}(t)) 
&\le \|((1-t)d_x + td_y) - ((1-t)d_{x'} + td_{y'})\|_\infty \\
&\le (1-t) \|d_x - d_{x'}\|_\infty + t \|d_y - d_{y'}\|_\infty \\
&= (1-t)d(x,x') + t d(y,y') 
\end{align*}
for $x,y,x',y' \in X$ and $t \in [0,1]$. Similarly,
\begin{align*}
d(\gam_{xy}(s),\gam_{xy}(t))
&\le \|((1-s)d_x + sd_y) - ((1-t)d_x + td_y)\|_\infty \\
&= (t-s) \|d_x - d_y\|_\infty \\
&= (t-s) d(x,y) 
\end{align*}
for $x,y \in X$ and $0 \le s \le t \le 1$, and it is easy to see that
equality must hold since $\gam_{xy}(0) = x$ and $\gam_{xy}(1) = y$. 
Thus $\gam$ is a geodesic bicombing on $X$ that satisfies~(1) and~(2).

Now let $L \in \Isom(X)$, and recall that 
$\bar{L}(p(f)) = p(f \circ L^{-1})$ for 
all $f \in \D{X}$, by Proposition~\ref{Prop:isometries}. 
Since $d_v \circ L^{-1} = d_{L(v)}$ for all $v \in X$, we have
\[
\bar{L}(p((1-t)d_x + td_y)) = p((1-t)d_{L(x)} + t d_{L(y)})
\]
for $x,y \in X$ and $t \in [0,1]$. As $\e^{-1} \circ \bar{L} = L \circ \e^{-1}$,
this gives~(3).
\end{proof}


\section{Polyhedral structure} \label{Sect:poly}

The main purpose of this section is to show that under suitable discreteness 
and local finite dimensionality assumptions on the metric space $X$ and its 
injective hull $\E{X}$, respectively,
the latter has the structure of a polyhedral complex, like for finite $X$.
Some results of this type are developed in~\cite[Sections~5 and~6]{Dre}. 
Here we give an independent treatment, with some emphasis on the analysis of 
the isometry types of cells.
For simplicity, we shall focus on integer valued metrics.

At first, let $X$ be an arbitrary metric space.
For $f \in \R^X$, we denote by $A(f)$ the set of all unordered 
pairs $\{x,y\}$ of points in $X$ with the property that 
\[
f(x) + f(y) = d(x,y).
\] 
We consider the undirected graph $(X,A(f))$ with vertex set $X$, 
edge set $A(f)$, and with loops $\{x,x\} \in A(f)$ marking the zeros of $f$.
If $f \in \D{X}$ and $X$ is finite (or compact), 
then $f$ is extremal if and only if $(X,A(f))$ has no isolated vertices, 
that is, $\bigcup A(f) = X$. For an infinite $X$, this need no longer 
be true. We therefore introduce the subset
\[
\Ex{X} := \bigl\{ f \in \D{X} : \textstyle\bigcup A(f) = X \bigr\}
\]
of $\E{X}$, whose structure can be analyzed more directly,
but which is not injective unless it coincides with $\E{X}$.
(In~\cite{Dre}, $\Ex{X}$ is denoted $T_X^0$.) 
Proposition~\ref{Prop:e-dense} below will show that $\Ex{X}$ is 
dense in $\E{X}$ in case the metric of $X$ is integer valued.

A set $A$ of unordered pairs of points in $X$ is called 
an {\em admissible} edge set if there exists a function $f \in \Ex{X}$ 
with $A(f) = A$, and $\cA(X)$ denotes the set of all such admissible 
sets. Let $A \in \cA(X)$. Note that the graph $(X,A)$ has no isolated 
vertices but need not be connected. 
We associate with $A$ the affine subspace
\begin{align*}
H(A) &:= \bigl\{ g \in \R^X : A \sub A(g) \bigr\} \\
&\phantom{:}= \bigl\{ g \in \R^X : \text{$g(x) + g(y) = d(x,y)$ 
for all $\{x,y\} \in A$} \bigr\}
\end{align*}
of $\R^X$, and we define the {\em rank} of $A$
as the dimension of $H(A)$,
\[
\rk(A) := \dim(H(A)) \in \{0,1,2,\dots\} \cup \{\infty\}.
\]
An \emph{$A$-path} in $X$ of length $l \ge 0$ is an $(l+1)$-tuple 
$(v_0,\dots,v_l) \in X^{l+1}$ with $\{v_{i-1}, v_i\} \in A$
for $i = 1,\dots,l$. An \emph{$A$-cycle} is an $A$-path $(v_0,\dots,v_l)$ 
with $v_l = v_0$. Note that $(x,x)$ is an $A$-cycle of length~$1$ 
if $\{x,x\} \in A$. The {\em $A$-component} 
$\vc{x}$ of a point $x \in X$ is the set
\[
\vc{x} := \{y \in X: \text{there exists an $A$-path from $x$ to $y$}\}.
\]
Whenever $g,h \in H(A)$ and 
$\{v,v'\}\in A$, we have $g(v) + g(v') = d(v,v') = h(v) + h(v')$ 
and so $g(v') - h(v') = -(g(v) - h(v))$. 
It follows that
\begin{equation} \label{eq:g-h}
  g(y) - h(y) = (-1)^{l}(g(x) - h(x)) 
\end{equation}
whenever there is an $A$-path of length $l$ from $x$ to $y$.
As a consequence, if there exists an $A$-cycle of odd length in $\vc{x}$,
then $g|_{\vc{x}} = h|_{\vc{x}}$ for all $g,h \in H(A)$.
We call $\vc{x}$ an {\em odd $A$-component} in this case.
In the opposite case, if $\vc{x}$ contains no $A$-cycle of odd length, 
$\vc{x}$ is called an {\em even $A$-component}.
Then the set $\{ g|_{\vc{x}} : g \in H(A) \}$ forms a one-parameter family. 
In fact, every even $A$-component admits a unique partition 
\begin{equation} \label{eq:par}
\vc{x} = \vc{x}_1 \cup \vc{x}_{-1}
\end{equation}
such that $x \in \vc{x}_1$ and every edge $\{v,v'\} \in A$ with
$\{v,v'\} \sub \vc{x}$ connects $\vc{x}_1$ and $\vc{x}_{-1}$; that is, 
the subgraph of $(X,A)$ induced by $\vc{x}$ is bipartite.
Then, by~\eqref{eq:g-h}, $g(y) - h(y) = \sig(g(x) - h(x))$ whenever 
$g,h \in H(A)$, $\sig \in \{1,-1\}$, and $y \in \vc{x}_\sig$.
It is now clear that $\rk(A)$ is exactly the number of even 
$A$-components of $X$. 
If $\rk(A) = 0$, $H(A)$ consists of a single function. 
This occurs in particular if $A = A(d_x)$ for some $x\in X$; 
then $\{x,y\} \in A$ for every $y \in X$, so $X$
is $A$-connected, and $(x,x)$ is an $A$-cycle of length $1$.

\begin{Lem} \label{Lem:ha}
Suppose that $X$ is a metric space, $A \in \cA(X)$, and 
$1 \le n := \rk(A) < \infty$.
Then the difference of any two elements of $H(A)$ is uniformly
bounded on $X$, so the supremum norm gives a metric on $H(A)$,
and there exists an affine isometry from $H(A)$ onto $l^n_\infty$.
In particular $H(A)$ is injective.
\end{Lem}

\begin{proof}
Choose reference points $x_1,\dots,x_n \in X$ such that 
$\vc{x_1},\dots,\vc{x_n}$ are precisely the $n$ even $A$-components 
of~$X$. 
Let $I \colon H(A) \to l^n_\infty$ be the affine map defined by
\[
I(g) := (g(x_1),\dots,g(x_n)).
\]  
It follows from~\eqref{eq:g-h} that 
$\|g - h\|_\infty = \max_{1 \le k \le n} |g(x_k) - h(x_k)|
= \|I(g) - I(h)\|_\infty$ for all $g,h \in H(A)$.
\end{proof}

For every $A \in \cA(X)$ we consider the set
\[
P(A) := \Ex{X} \cap H(A) = \{g \in \Ex{X} : A \sub A(g)\}.
\]
First we note that
\begin{equation} \label{eq:pg-char}
P(A) = \E{X} \cap H(A)
= \D{X} \cap H(A).
\end{equation}
To see this, let $f \in \Ex{X}$ be such that $A(f) = A$, and 
let $g \in \D{X} \cap H(A)$. Every $x \in X$ is part of an edge 
$\{x,y\} \in A(f) = A$; then $\{x,y\} \in A(g)$ because $g \in H(A)$. 
Since $g \in \D{X}$, this shows that $g \in \Ex{X}$. In view of the inclusions
$\Ex{X} \sub \E{X} 
\sub \D{X}$ we get~\eqref{eq:pg-char}. As $\D{X}$ is convex,
so is $P(A)$. For every $f \in \Ex{X}$ we have $f \in P(A(f))$, thus
\[
\cP := \{P(A)\}_{A \in \cA(X)}
\]
is a family of convex subsets of $\R^X$ whose union equals $\Ex{X}$. Note that 
$P(A') \sub P(A)$ if and only if $A \sub A'$.
The next result lists some basic properties of $\cA(X)$ and $\cP$.

\begin{Prop} \label{Prop:p-properties}
Let $X$ be a metric space. Then:
\begin{enumerate}
\item[\rm (1)]
If $f_0,f_1 \in \Ex{X}$ and $\lam \in (0,1)$, then 
$f := (1-\lam)f_0 + \lam f_1 \in \D{X}$ and $A(f) = A(f_0) \cap A(f_1)$, 
so $f \in \Ex{X}$ if and only if $\bigcup(A(f_0) \cap A(f_1)) = X$.
\item[\rm (2)]
For $A_0,A_1 \in \cA(X)$, the following are equivalent:
\begin{enumerate}
\item[\rm (i)] 
$P(A_0) \cup P(A_1) \sub P(A)$ for some $A \in \cA(X)$; 
\item[\rm (ii)] 
$\bigcup(A_0 \cap A_1) = X$; 
\item[\rm (iii)] 
$A_0 \cap A_1 \in \cA(X)$.
\end{enumerate}
If conditions {\rm (i)--(iii)} hold, 
then $P(A_0) \cup P(A_1) \sub P(A_0 \cap A_1) \sub
P(A)$.
\end{enumerate}
\end{Prop}

\begin{proof}
Let $f$ be given as in~(1). 
Since $f_0,f_1 \in \D{X}$ and $\lam \in (0,1)$, we have
\begin{align*}
f(x) + f(y) 
&= (1-\lam)(f_0(x) + f_0(y)) + \lam(f_1(x) + f_1(y)) \\
&\ge (1-\lam)d(x,y) + \lam d(x,y) = d(x,y)
\end{align*}
for every pair of points $x,y \in X$, and 
equality holds if and only if $\{x,y\} \in A(f_0) \cap A(f_1)$.

Regarding~(2), choose $f_0,f_1 \in \Ex{X}$ such that $A(f_i) = A_i$, and put
$f := \frac12(f_0 + f_1)$. 
If $P(A_0) \cup P(A_1) \sub P(A)$ for some $A \in \cA(X)$, 
then $f \in P(A)$ by convexity and hence $f \in \Ex{X}$, 
so $\bigcup(A_0 \cap A_1) = X$ by~(1). 
If~(ii) holds, then, again by~(1), $f \in \Ex{X}$ and so
$A_0 \cap A_1 = A(f) \in \cA(X)$. 
Finally, assuming~(iii), we obtain
$P(A_0) \cup P(A_1) \sub P(A_0 \cap A_1)$ and thus~(i), and 
for every $A \in \cA(X)$ with $P(A_0) \cup P(A_1) \sub P(A)$
we have $A \sub A_0 \cap A_1$ and hence $P(A_0 \cap A_1) \sub P(A)$.
\end{proof}

We now pass to integer valued metrics. Then the sets $P(A)$ with 
$\rk(A) = n < \infty$ turn out to be $n$-dimensional polytopes:

\begin{Thm} \label{Thm:pa}
Suppose that $X$ is a metric space with integer valued metric.
Let $A \in \cA(X)$, and assume that $1 \le n := \rk(A) < \infty$. Then:
\begin{enumerate}
\item[\rm (1)]
The set $P(A) \sub H(A) \sub \R^X$ is an injective $n$-dimensional
polytope.
\item[\rm (2)]
The interior of $P(A)$ relative to $H(A)$ is the set
$\{g \in \Ex{X} : A(g) = A\}$. 
\item[\rm (3)]
The faces of $P(A)$ are precisely the 
sets $P(A')$ with $A' \in \cA(X)$ and $A \sub A'$.
\end{enumerate}
\end{Thm}

The proof will also give precise information on the 
possible isometry types of~$P(A)$.

\begin{proof}
We fix reference points $x_1,\dots,x_n \in X$ representing the 
$n$ even $A$-components. For each $k$ we consider the partition 
$\vc{x_k} = \vc{x_k}_1 \cup \vc{x_k}_{-1}$ as in~\eqref{eq:par}. 
We also fix an element $f \in \Ex{X}$ with $A(f) = A$. 
For $y \in \vc{x_k}_\sig$, $\sig \in \{1,-1\}$, we have
\begin{equation} \label{eq:f-even}
f(y) \in \Z + \sig f(x_k).
\end{equation}
By contrast, if $y \in X_0 := X \sm \bigcup_{k=1}^n \vc{x_k}$,
then there is an $A$-path from $y$ to itself of odd length, so
$f(y) \in \Z - f(y)$ and thus
\begin{equation} \label{eq:f-odd}
f(y) \in \frac12 \Z.
\end{equation}
Now let $I_f \colon H(A) \to l^n_\infty$ be the affine isometry defined by
\[
I_f(g) := (g(x_1) - f(x_1),\dots,g(x_n) - f(x_n));
\]
compare the proof of Lemma~\ref{Lem:ha}. 
To show that $I_f(P(A))$ is a polytope,
we introduce constants as follows. 
First, for $1 \le k \le n$ and $\sig \in \{1,-1\}$, put
\[
C_{k \sig} := \sup \biggl\{ \frac{d(x,y) - f(x) - f(y)}{2} :
x,y \in \vc{x_k}_\sig \biggl\}.
\]
For $x,y \in \vc{x_k}_\sig$ we have $\{x,y\} \not\in A = A(f)$,
hence $0 > d(x,y) - f(x) - f(y) \in \Z - 2 \sig f(x_k)$ by~\eqref{eq:f-even}.
Thus the supremum is attained, and
$C_{k \sig} < 0$.   
Next, if $X_0 \ne \es$, then for $1 \le k \le n$ and $\sig \in \{1,-1\}$,
define
\[
C_{k \sig 0} := \sup \bigl\{ d(x,y) - f(x) - f(y) :
(x,y) \in \vc{x_k}_\sig \times X_0 \bigr\}.
\]
For every such pair $(x,y)$, we have
$0 > d(x,y) - f(x) - f(y) \in \frac12\Z - \sig f(x_k)$, hence
$C_{k \sig 0} < 0$.
Set $\bar C_{k \sig} := C_{k \sig}$ if $X_0 = \es$ and
\[
\bar C_{k \sig} := \max \bigl\{ C_{k \sig},C_{k \sig 0} \bigr\}
\] 
if $X_0 \ne \es$. 
Finally, if $n \ge 2$, then for $1 \le k < l \le n$ and 
$\sig,\tau \in \{1,-1\}$, define
\[
C_{k \sig l \tau} := \sup \bigl\{ d(x,y) - f(x) - f(y) :
(x,y) \in \vc{x_k}_\sig \times \vc{x_l}_\tau \bigr\}.
\]
For every such pair $(x,y)$, we have $0 > d(x,y) - f(x) - f(y) 
\in \Z - \sig f(x_k) - \tau f(x_l)$ and so $C_{k \sig l \tau} < 0$.
Now let $Q$ denote the set of all
$t = (t_1,\dots,t_n) \in l^n_\infty$ satisfying the system of 
$2n + 4\binom{n}{2} = 2n^2$ relations
\begin{align}
\sig t_k &\ge \bar C_{k \sig \phantom{l \tau}} 
\quad \text{($1 \le k \le n$, $\,\sig \in \{1,-1\}$),} \label{eq:def-q-1} \\
\sig t_k + \tau t_l &\ge C_{k \sig l \tau} 
\quad \text{($1 \le k < l \le n$, $\,\sig,\tau \in \{1,-1\}$).} 
\label{eq:def-q-2}
\end{align}
By the first $2n$ inequalities $Q$ is bounded.
Since all constants on the right side are strictly negative, $Q$ is a polytope 
containing $I_f(f) = 0$ in its interior, so $Q$ has dimension $n$.
It follows readily from Proposition~\ref{Prop:p-inj} 
that $Q$ is itself injective.

We claim that $I_f(P(A)) = Q$. Let $g \in H(A)$ and $t := I_f(g)$. 
In view of~\eqref{eq:pg-char}, we need to check that $t \in Q$ if and only if 
$g(x) + g(y) \ge d(x,y)$ for all pairs $\{x,y\} \not\in A$.
First, consider pairs of points $x,y \in \vc{x_k}$, for some 
$k$. If $\sig \in \{1,-1\}$ and $x,y \in \vc{x_k}_\sig$, then
\[
2\sig t_k = 2\sig(g(x_k) - f(x_k)) = g(x) - f(x) + g(y) - f(y)
\]
by~\eqref{eq:g-h}. Hence, we have $\sig t_k \ge C_{k \sig}$ 
if and only if the inequality $g(x) + g(y) \ge d(x,y)$ 
holds for all pairs of points $x,y \in \vc{x_k}_\sig$.
If $x \in \vc{x_k}_1$ and $y \in \vc{x_k}_{-1}$, then $g(x) + g(y) = f(x) + f(y) 
> d(x,y)$ by~\eqref{eq:g-h} and since $\{x,y\} \not\in A$ by assumption. 
Next, in case $X_0 \ne \es$, consider pairs 
$(x,y) \in \vc{x_k}_\sig \times X_0$, for some $k$ and 
$\sig$. Then $g(y) = f(y)$ and so 
\[
\sig t_k = \sig(g(x_k) - f(x_k)) = g(x) - f(x) + g(y) - f(y),
\]
hence $\sig t_k \ge C_{k \sig 0}$ 
if and only if $g(x) + g(y) \ge d(x,y)$ 
for all such $(x,y)$. 
If $x,y \in X_0$ and $\{x,y\} \not\in A$, 
then $g(x) + g(y) = f(x) + f(y) > d(x,y)$.
Finally, in case $n \ge 2$, consider pairs 
$(x,y) \in \vc{x_k}_\sig \times \vc{x_l}_{\tau}$ for some 
$k < l$ and $\sig,\tau$. Then
\[
\sig t_k + \tau t_l = \sig(g(x_k) - f(x_k)) + \tau(g(x_l) - f(x_l))
= g(x) - f(x) + g(y) - f(y),
\]
therefore $\sig t_k + \tau t_l \ge C_{k \sig l \tau}$ 
if and only if $g(x) + g(y) \ge d(x,y)$ 
for all such $(x,y)$. 
This yields $I_f(P(A)) = Q$ and completes the proof of~(1).

Since $f$ was an arbitrary element of $\Ex{X}$ with $A(f) = A$
and $I_f(f)$ is an inner point of $Q$, it follows
that the set $\{g \in \Ex{X} : A(g) = A\}$ is contained in the 
relative interior of $P(A)$. 
Furthermore, if $g \in P(A)$ is such that the inclusion $A \sub A(g)$ 
is strict, then $g(x) + g(y) = d(x,y)$ for some pair $\{x,y\} \not\in A$ 
and we see from the above argument that equality holds in at least one 
of the $2n^2$ inequalities~\eqref{eq:def-q-1}, \eqref{eq:def-q-2};
thus $I_f(g)$ is a boundary point of $Q$. This shows~(2).

Now suppose that $F$ is face of $P(A)$ of dimension $n-1$. 
Choose a point $g$ in the relative interior of $F$.
Since $g \in H(A)$, we have $H(A(g)) \sub H(A)$. 
For $t := I_f(g)$, exactly one of the $2n^2$ inequalities~\eqref{eq:def-q-1}, 
\eqref{eq:def-q-2} is an equality and the others are strict. 
Reviewing the above argument again we see that then 
the inclusion $A \sub A(g)$ is strict and exactly one of 
the following two cases occurs: 
\begin{enumerate}
\item[(i)]
there exist $k$ and $\sig$ such that every 
edge in $A(g) \setminus A$ relates two (possibly equal) points 
of $\vc{x_k}_\sig$ or connects $\vc{x_k}_\sig$ with $X_0$;
\item[(ii)]
there exist $k < l$ and $\sig,\tau$ such that every edge in
$A(g) \setminus A$ connects $\vc{x_k}_\sig$ with $\vc{x_l}_\tau$. 
\end{enumerate}
In either case, $X$ has $n-1$ even $A(g)$-components, thus $H(A(g))$ 
is an $(n-1)$-dimensional affine subspace of $H(A)$.
For all $h \in H(A(g))$, $A \sub A(g) \sub A(h)$ and the first 
inclusion is strict, so $H(A(g))$ contains no inner points of $P(A)$ 
by~(2). As $g \in H(A(g))$ is in the relative interior of $F$,
we have $F = P(A) \cap H(A(g))$ and hence $F = P(A(g))$.
Now it follows easily by downward induction on $k$ 
that every face $F$ of $P(A)$ of dimension $k \in \{0,\dots,n\}$
satisfies $F = P(A_F)$ for some $A_F \in \cA(X)$ with $A \sub A_F$ 
and $\rk(A_F) = k$. Conversely, let $A' \in \cA(X)$ with 
$A \sub A'$ be given. Then $P(A') \sub P(A)$, so the relative interior 
of $P(A')$ meets the relative interior of some face $F = P(A_F)$ of 
$P(A)$. Applying~(2) to both $P(A')$ and $P(A_F)$ we obtain 
$A' = A_F$, thus $P(A') = P(A_F) = F$.
This concludes the proof of~(3).
\end{proof}

Next we show that $\Ex{X}$ is dense in $\E{X}$, provided the metric of $X$ 
is integer valued. A different criterion is given in~\cite[(5.17)]{Dre}.
 
\begin{Prop} \label{Prop:e-dense}
Let $X$ be a metric space with integer valued metric.
Then for every $f \in \E{X}$ and every integer $m \ge 1$ there 
exists a function $f' \in \Ex{X}$ with values in $\frac1m\Z$ such that 
$\|f - f'\|_\infty \le \frac{1}{2m}$. 
\end{Prop}

\begin{proof}
Let $f \in \E{X}$, $m \ge 1$, and put $\eps := \frac{1}{2m}$.
Denote by $\cF$ the set of all functions $g \in \D{X}$ with values 
in $2\eps\Z$ and with $\|f - g\|_\infty \le \eps$. To see that
$\cF$ is non-empty, let $g_0 \in \R^X$ be the largest function less 
than or equal to $f + \eps$ with values in $2\eps\Z$. 
Then $g_0 > f - \eps$, in particular $\|f - g_0\|_\infty \le \eps$.
For $x,y \in X$, we have $g_0(x) + g_0(y) > f(x) + f(y) - 2\eps 
\ge d(x,y) - 2\eps$, and since both the first and the last term
are in $2\eps\Z$, this gives $g_0(x) + g_0(y) \ge d(x,y)$. So $g_0 \in \cF$.

Now let $g \in \cF$ be arbitrary, and suppose that $x \in X \sm \bigcup A(g)$.
Then, for every $y \in X$, we have the strict inequality 
$g(x) > d(x,y) - g(y)$ in $2\eps\Z$, so that
\[
g(x) \ge \sup_{y \in X}(d(x,y) - g(y) + 2\eps) 
\ge \sup_{y \in X}(d(x,y) - f(y) + \eps) = f(x) + \eps
\]
by~\eqref{eq:extr-sup}. Hence $g(x) = f(x) + \eps$.
Let the function $g'$ be defined by
\[
g'(x) := g(x) - 2\eps = f(x) - \eps
\]
and $g'(y) := g(y)$ for all $y \in X \sm \{x\}$. Note that $g(x) \ge \eps$,
thus in fact $g(x) \ge 2\eps$ and $g'(x) \ge 0$.
Since $g'(x) + g'(y) = g(x) + g(y) - 2\eps \ge d(x,y)$ for all 
$y \in X \sm \{x\}$, it follows that $g' \in \cF$. This shows that
every minimal element $f'$ of $\cF$ satisfies $\bigcup A(f') = X$,
that is, $f' \in \Ex{X}$. The existence of some minimal element is 
obvious if $X$ is countable and a consequence of Zorn's Lemma in the 
general case. 
\end{proof}

We now state the concluding result of this section.
A metric space $X$ with integer valued metric will be called 
{\em discretely path-connected} if for every pair of points 
$x,y \in X$ there exists a {\em discrete path} 
$\gam \colon \{0,1,\dots,l\} \to X$ from $x$ to $y$, that is, 
$\gam(0) = x$, $\gam(l) = y$, and $d(\gam(k-1),\gam(k)) = 1$ 
for $k = 1,\dots,l$. 

\begin{Thm} \label{Thm:poly}
Let $X$ be a metric space with integer valued metric.
Suppose that for every $f \in \E{X}$ there exist
$\eps,N > 0$ such that $\rk(A(g)) \le N$ for
all $g \in \Ex{X}$ with $\|f - g\|_\infty < \eps$.
Then:
\begin{enumerate}
\item[\rm (1)]
$\Ex{X} = \E{X}$.
\item[\rm (2)]
$\cP = \{P(A)\}_{A\in \cA(X)}$ is a polyhedral structure on $\E{X}$ with 
locally finite dimension, 
where $P(A')$ is a face of $P(A)$ if and only if $A \sub A'$.
\item[\rm (3)]
For every $n \ge 1$ and $D > 0$, $\cP$ has only finitely many isometry
types of $n$-cells with diameter at most $D$.
If, in addition, $X$ is discretely path-connected,
then for every $n$ there are only finitely many 
isometry types of $n$-cells.
\end{enumerate}
\end{Thm}

\begin{proof}
For~(1), let $f \in \E{X}$.
By Proposition~\ref{Prop:e-dense} there exists a sequence
$(f_i)$ in $\Ex{X}$ that converges to $f$, and by the assumption of the 
theorem there is no loss of generality in assuming that $\rk(A(f_i)) = n$
for all $i$ and for some $n \ge 0$. It follows that for every $i$ there 
exists a set $R_i \sub [0,1)$ with $|R_i| \le 2n + 2$ such that $f_i$ takes 
values in $\Z + R_i$; see~\eqref{eq:f-even} and~\eqref{eq:f-odd}. 
Since $f_i \to f$, there also exists $R \sub [0,1)$ with $|R| \le 2n + 2$ 
such that $f(X) \sub \Z + R$.
But then the supremum in~\eqref{eq:extr-sup} is attained for every $x \in X$,
and so $f \in \Ex{X}$. (Compare~\cite[(5.19)]{Dre}.)

The union of the family $\cP = \{P(A)\}_{A \in \cA(X)}$ equals $\Ex{X} = \E{X}$. 
In view of Theorem~\ref{Thm:pa}, for~(2) it remains
to show that if $A_1,A_2 \in \cA(X)$ and 
$C := P(A_1) \cap P(A_2) \ne \es$, then
$C \in \cP$. For $i = 1,2$, let $P(A'_i)$ be the minimal face of 
$P(A_i)$ that contains $C$. By convexity, $C$ has non-empty interior 
relative to its affine hull in $\R^X$, hence the relative interiors of 
$P(A'_1)$ and $P(A'_1)$ have a common point. It follows that
$A'_1 = A'_2$ and thus $P(A'_1) = P(A'_2) = C$.

As for~(3), we first observe that if $f \in \Ex{X}$ is a vertex of $\cP$, 
then $\rk(A(f)) = 0$ and so $f(X) \sub \frac12\Z$ by~\eqref{eq:f-odd}.
In particular, all edges of $\cP$ have length in $\frac12\Z$.
Now we show that if $X$ is discretely path-connected,
then all edges have length at most $2$. 
Suppose that $A \in \cA(X)$, $\rk(A) = 1$, 
and $x_1$ is a point in the only even $A$-component of $X$. 
Then clearly there exists a pair $(x,y)$ with $d(x,y) = 1$ such that 
$x \in \vc{x_1}_1$ and either $y \in \vc{x_1}_{-1}$ 
or $y \in X \sm \vc{x_1}$. Let $g,h \in P(A)$. By~\eqref{eq:g-h},
\[
\|g - h\|_\infty = |g(x) - h(x)|.
\]
In case~$y \in \vc{x_1}_{-1}$, we have furthermore 
$g(x) - h(x) = -(g(y) - h(y))$;
since $g,h$ are $1$-Lipschitz and $d(x,y) = 1$, it follows that 
\begin{align*}
2|g(x) - h(x)| 
&= |g(x) - h(x) - (g(y) - h(y))| \\
&\le |g(x) - g(y)| + |h(y) - h(x)| \le 2. 
\end{align*}
In case~$y \in X \sm \vc{x_1}$, we have $g(y) = h(y)$ and so 
\[
|g(x) - h(x)| \le |g(x) - g(y)| + |h(y) - h(x)| \le 2.
\]
In either case, $\|g - h\|_\infty \le 2$.
Hence the edge $P(A)$ has length at most $2$.
Finally, for $n \ge 2$, we see from~\eqref{eq:def-q-1}, \eqref{eq:def-q-2} 
that there are only finitely many isometry types of 
$n$-cells with diameter at most $D > 0$ and edge lengths in $\frac12\Z$,
and only finitely many isometry types of $n$-cells with edge lengths in
$\{\frac12,1,\frac32,2\}$.
\end{proof}
 
For illustration, we give two simple examples of (finite) discretely 
geodesic metric spaces with two-dimensional injective hulls. 
Together they show in particular that the four possible lengths
$\frac12,1,\frac32,2$ for edges in $\cP$, as just discussed, do indeed occur.
For points $v_1,\dots,v_n$ in a vector space $V$ we will denote by
$[v_1,\dots,v_n]$ the convex hull $\{\sum_{i=1}^n \lam_iv_i : \lam_i \ge 0,\,
\sum_{i=1}^n \lam_i = 1\}$.

\begin{Expl}
Let $X = \{x_1,x_2,y_1,y_2,y_3\}$ be the 
metric space with $d(x_1,x_2) = 2$, $d(x_i,y_j) = 1$, 
and $d(y_j,y_k) = 2$ ($j < k$). 
Then $P' := [d_{x_1},d_{x_2}]$ is an edge of $\cP$ of length $2$. 
The maximal cells of the injective hull $\E{X}$ (that is, of the partially 
ordered set $\cP$) are the three triangles 
$P_j := [d_{x_1},d_{x_2},d_{y_j}]$, isometric to $[(1,0),(-1,0),(0,1)]$ in 
$l^2_\infty$. They are glued along $P'$. For instance, 
$A_1 := \{\{x_1,x_2\},\{y_1,y_2\},\{y_1,y_3\}\}$
and $A' := A_1 \cup \{\{y_2,y_3\}\}$ are the respective admissible
edge sets with $P(A_1) = P_1$ and $P(A') = P'$.
\end{Expl}

\begin{Expl}
Consider the six-point metric space $X = \{x_1,x_2,y_1,y_2,y_3,z\}$, 
where all distances between distinct points are $1$ except that 
$d(x_i,z) = 2$ and $d(y_j,y_k) = 2$ ($j < k$). 
Besides the distance functions to the elements
of $X$, the injective hull $\E{X}$ has four additional vertices 
$f_1,f_2,f_3,g$, where $f_j = \frac12$ on $\{x_1,x_2,y_j\}$, 
$f_j = \frac32$ otherwise, $g = \frac12$ on $\{x_1,x_2\}$, $g = 1$ on 
$\{y_1,y_2,y_3\}$,
and $g(z) = \frac32$. The graph $(X,A(g))$
consists of two $3$-cycles, with vertices $x_1,x_2,z$ and $y_1,y_2,y_3$,
respectively. By deleting one of the six edges of this graph one obtains
the graph corresponding to an edge of $\E{X}$ that connects $g$ with
one of $d_{x_1},d_{x_2},d_z,f_1,f_2,f_3$. All of these edges have length 
$\frac12$, except for $[g,d_z]$, which has length $\frac32$. 
For instance, deletion of $\{y_1,y_2\}$ gives the graph for $[g,f_3]$. 
The maximal cells of $\E{X}$ are the six triangles $[d_{x_i},f_j,g]$, 
isometric to $[(\frac12,0),(0,\frac12),(0,0)]$ in $l^2_\infty$,
and the three quadrilaterals $[g,f_j,d_{y_j},d_z]$, isometric to 
$[(0,0),(0,\frac12),(-\frac12,1),(0,-\frac32)]$ in $l^2_\infty$.
\end{Expl}


\section{Cones} \label{Sect:cones}

We now discuss geometric conditions that allow to verify the assumption on
the rank in Theorem~\ref{Thm:poly}. 
Cones, as defined in~\eqref{eq:c}, will be instrumental. 
We start with a basic fact.
 
\begin{Lem} \label{Lem:co}
Suppose that $X$ is a metric space, $f \in \Dl{X}$, and $\{x,y\} \in A(f)$.
Then $\{x,z\} \in A(f)$ and $f(z) = f(y) + d(y,z)$ for all $z \in \co(x,y)$. 
\end{Lem}

\begin{proof}
For $f \in \D{X}$ and $z \in \co(x,y)$, we have
\[
f(x) + f(z) \ge d(x,z) = d(x,y) + d(y,z).
\]
Furthermore, if $f \in \Lip_1(X,\R)$ and $\{x,y\} \in A(f)$, then
\[
f(x) + f(z) \le f(x) + f(y) + d(y,z) = d(x,y) + d(y,z).
\]
This gives the result.
\end{proof}

The next lemma, in particular criterion~(4), 
will play a key role in the proof of Theorem~\ref{Thm:intro-ex}.
(For~(2), compare~\cite[Theorem~3.12]{GooM}.)
 
\begin{Lem} \label{Lem:xyxy}
Let $X$ be a metric space, and suppose 
that $f \in \D{X}$, $x,y,\bar x,\bar y \in X$, and 
$\{x,y\},\{\bar x,\bar y\} \in A(f)$. Then each of the following
conditions implies that also $\{x,\bar y\},\{\bar x,y\} \in A(f)$:
\begin{enumerate}
\item[\rm (1)]
$d(x,y) + d(\bar x,\bar y) \le d(x,\bar y) + d(\bar x,y)$;
\item[\rm (2)]
$\co(x,y) \cap \co(\bar x,\bar y) \ne \es$;
\item[\rm (3)]
$\bw(x,\bar y) \cap \bw(\bar x,y) \ne \es$;
\item[\rm (4)]
there exists $v \in \bw(x,y) \cap \bw(\bar x,\bar y)$ 
such that $\co(x,v) = \co(\bar x,v)$.
\end{enumerate}
\end{Lem}


\begin{proof}
Because $\{x,y\},\{\bar x,\bar y\} \in A(f)$, (1) gives 
\[
f(x) + f(y) + f(\bar x) + f(\bar y) 
= d(x,y) + d(\bar x,\bar y)
\le d(x,\bar y) + d(\bar x,y).
\]
It follows that each of the inequalities $f(x) + f(\bar y) \ge d(x,\bar y)$ 
and $f(\bar x) + f(y) \ge d(\bar x,y)$ must in fact be an equality, that is,
$\{x,\bar y\},\{\bar x,y\} \in A(f)$.
Now assume that~(2) holds, and let $z \in \co(x,y) \cap \co(\bar x,\bar y)$.
Then 
\[
d(x,y) + d(y,z) = d(x,z) \le d(x,\bar y) + d(\bar y,z)
\]
and, likewise, $d(\bar x,\bar y) + d(\bar y,z) \le  d(\bar x,y) + d(y,z)$. 
Adding these two inequalities one obtains~(1).
If $v \in \bw(x,\bar y) \cap \bw(\bar x,y)$, then
\[
d(x,y) + d(\bar x,\bar y)
\le d(x,v) + d(v,y) + d(\bar x,v) + d(v,\bar y)
= d(x,\bar y) + d(\bar x,y),
\]
so~(3) implies~(1) as well. Finally, if (4) holds, then 
$\bar y \in \co(\bar x,v) = \co(x,v)$ and $y \in \co(x,v) = \co(\bar x,v)$, 
thus $v \in \bw(x,\bar y) \cap \bw(\bar x,y)$.
\end{proof}

As a first simple application of these lemmas we note the following 
result.

\begin{Prop} \label{Prop:disj-cones}
Suppose that $X$ is a metric space containing at most $k$ pairwise disjoint
cones, that is, $|I| \le k$ for every disjoint family 
$(\co(x_i,y_i))_{i \in I}$ of cones in $X$.
Then $\rk(A) \le \frac12 k$ for all $A \in \cA(X)$. 
\end{Prop}

\begin{proof}
Let $A \in \cA(X)$, and suppose that the two edges 
$\{x,y\},\{\bar x,\bar y\} \in A$ belong to different even 
$A$-components of $X$. It follows from either Lemma~\ref{Lem:co} or 
Lemma~\ref{Lem:xyxy} that the four cones 
$\co(x,y),\co(y,x),\co(\bar x,\bar y),\co(\bar y,\bar x)$ 
are pairwise disjoint.
For instance, if there was a point $z$ in $\co(x,y) \cap \co(\bar x,\bar y)$, 
then $\{x,z\},\{\bar x,z\} \in A$ by the first and 
$\{x,\bar y\},\{\bar x,y\} \in A$ by the second lemma, 
so $x$ and $\bar x$ would be connected by $A$-paths of length~2. 
This clearly gives the result.
\end{proof}
 
An example will be given below. First we record
another useful fact related to cones. 

\begin{Prop} \label{Prop:y-x}
Let $Y$ be a metric space, and let $X$ be a non-empty subset.
If for every pair of points $x,y \in Y$ there exists 
a point $z \in \co(x,y) \cap X$, then $\E{Y}$ is isometric to $\E{X}$
via the restriction map $f \mapsto f|_X$.
\end{Prop}

\begin{proof}
Let $f \in \E{Y}$, and let $x \in Y$. For every 
$\eps > 0$ there exist $y \in Y$ and $z \in X$ such that 
$f(x) + f(y) \le d(x,y) + \eps$ and $d(x,y) + d(y,z) = d(x,z)$;
since $f(z) \le f(y) + d(y,z)$, this gives 
$f(x) + f(z) \le d(x,z) + \eps$. Hence
\begin{equation} \label{eq:sup-z-x}
f(x) = \sup_{z \in X}(d(x,z) - f(z)).
\end{equation}
For $x \in X$, this shows that $f|_X \in \E{X}$.
Furthermore, for another function  $g \in \E{Y}$, 
combining~\eqref{eq:sup-z-x} with the inequality $d(x,z) \le g(x) + g(z)$ 
we conclude that $f(x) - g(x) \le \sup_{z \in X}(g(z) - f(z))$
for every $x \in Y$. So $\|f-g\|_\infty = \|f|_X - g|_X\|_\infty$,
and by the second part of Proposition~\ref{Prop:x-subspace}, 
the restriction operator $f \mapsto f|_X$ maps $\E{Y}$ onto $\E{X}$.
\end{proof}

We now illustrate Proposition~\ref{Prop:disj-cones}, 
which turns out to be optimal in some instances.

\begin{Expl} \label{Expl:zn}
Consider the discretely geodesic metric space $X = \Z^n$ with the 
$l_1$ distance (the standard word metric of the group $\Z^n$). 
It is not difficult to see that $X$ contains at most $2^n$ pairwise 
disjoint cones.
By Theorem~\ref{Thm:poly} and Proposition~\ref{Prop:disj-cones}, 
$\E{X}$ is a polyhedral complex of dimension at most $2^{n-1}$.
For the subspace $W_n := \{0,1\}^n \sub X$ of diameter $n$, 
the constant function $g$ on $W_n$ with value $\frac12 n$ satisfies 
$g \in \E{W_n}$ and $\rk(A(g)) = 2^{n-1}$ (each pair of antipodal points is an 
$A(g)$-component of $W_n$); thus $\dim(\E{W_n}) = \dim(\E{X}) = 2^{n-1}$. 
Furthermore, it follows easily from Proposition~\ref{Prop:y-x} that 
$\E{X}$ is isometric to $\E{l^n_1}$. So $\E{X}$ is also a Banach space, 
isometric to $l^{2^{n-1}}_\infty$ (see Remark~\ref{Rem:banach}).
Unless $n = 1,2$, the dimension of $\E{X}$
is strictly larger than $n$ and hence the canonical action of $\Z^n$ on
$\E{X}$ is not cocompact. This can be remedied by taking 
the $l_\infty$ distance on $\Z^n$ instead (which is again a word metric);
then clearly the injective hull is isometric to $l^n_\infty$.
\end{Expl}

Given a metric space $X$ and a point $v \in X$,
we denote by $\cC(v)$ the set of all cones $\co(x,v)$ for $x \in X$.
The following result shows that if $\cC(v)$ happens to be finite,
then one obtains some control on the complexity of $\Ex{X}$ near $d_v$.
Note that here $X$ is not assumed to be discrete.  

\begin{Prop} \label{Prop:cv-finite}
Suppose that $X$ is a metric space and 
$v \in X$ is a point with $|\cC(v)| < \infty$. 
Consider the set $A_v := A(d_v) \in \cA(X)$. Then:
\begin{enumerate}
\item[\rm (1)]
Every admissible set $A \in \cA(X)$ with $A \sub A_v$ 
satisfies $\rk(A) \le \frac12|\cC(v)|$.
\item[\rm (2)]
There are at most $2^{|\cC(v)|-1} - 1$ sets $A \in \cA(X)$ with $A \sub A_v$
and $\rk(A) = 1$.
\end{enumerate}
\end{Prop}

Note that $\{x,y\} \in A_v = A(d_v)$ if and only if 
$d_v(x) + d_v(y) = d(x,y)$, that is, $v \in \bw(x,y)$.  

\begin{proof}
Let $A \sub A_v$ be admissible. There exists a partition
\begin{equation} \label{eq:x-partition}
X = X_0 \cup \bigcup_{j \in J}(X_{j,1} \cup X_{j,-1}),
\end{equation}
where $X_0$ is the union of all odd $A$-components
of $X$, $\{X_j\}_{j \in J}$ is the family of all even $A$-components,
and the partition $X_j = X_{j,1} \cup X_{j,-1}$ is such that 
no edge in $A$ relates points in the same subset.
Let $x,\bar x \in X$. There exist $y,\bar y$ such that 
$\{x,y\},\{\bar x,\bar y\} \in A \sub A_v$ and thus 
$v \in \bw(x,y) \cap \bw(\bar x,\bar y)$. In case $\co(x,v) = \co(\bar x,v)$  
it follows from Lemma~\ref{Lem:xyxy} that 
$\{x,\bar y\},\{\bar x,y\} \in A$, in particular $x$ and $\bar x$ are 
connected by an $A$-path of length~$2$. Hence, if $x$ and $\bar x$
lie in different sets of the above partition, then 
$\co(x,v) \ne \co(\bar x,v)$. Thus the number of even $A$-components 
is in fact finite and less than or equal to $\frac12|\cC(v)|$.

For the proof of~(2), let $\cA_v'$ denote the set of all $A \in \cA(X)$ 
with $A \sub A_v$ and $\rk(A) = 1$. We show that there is an injective map
$S$ from $\cA_v'$ into the set of all non-empty subsets of
$\cC(v)$ that do not contain $\co(v,v) = X$. 
For each $A \in \cA_v'$ there is a unique 
partition $X = X_0 \cup X_1 \cup X_{-1}$ such that every $g \in \Ex{X}$ with 
$A(g) = A$ satisfies
\begin{equation} \label{eq:g-dv}
g(x) = d_v(x) + \sig \|g-d_v\|_\infty 
\end{equation} 
for $\sig \in \{0,1,-1\}$ and $x \in X_\sig$. 
Note that every such $g$ is strictly positive since $\rk(A) > 0$, therefore 
$v \in X_1$ and thus $\co(x,v) \ne X = \co(v,v)$ for all $x \in X_{-1}$.
The desired map $S$ is defined by 
\[
S(A) := \{\co(x,v) : x \in X_{-1}\}.
\]
To show that $S$ is injective, suppose that $S(A) = S(A')$ for 
some $A,A' \in \cA_v'$, and let $X_0 \cup X_1 \cup X_{-1}$ and
$X'_0 \cup X'_1 \cup X'_{-1}$ be the respective partitions of $X$.
Now note, first, that 
\[
X_{-1} = \{x \in X : \co(x,v) \in S(A)\}.
\] 
This holds since $\co(\bar x,v) \ne \co(x,v)$ for 
$\bar x \in X_0 \cup X_1$ and $x \in X_{-1}$, 
by the same argument as in the proof of~(1). Similarly,
$X'_{-1} = \{x \in X : \co(x,v) \in S(A')\}$ and so $X_{-1} = X'_{-1}$.
Second, 
\[
X_{1} = \{y \in X : \text{there exists $x \in X_{-1}$ with 
$\{x,y\} \in A_v$}\}.
\]
The inclusion $\sub$ is clear since $A \sub A_v$.  
For the other, if $x \in X_{-1}$ and $\{x,y\} \in A_v$, 
then every $g \in \Ex{X}$ with $A(g) = A$ satisfies
$g(y) \ge d(x,y) - g(x) = d_v(x) + d_v(y) - g(x) = d_v(y) + \|g-d_v\|_\infty$,
so $y \in X_1$. Together with the corresponding characterization of 
$X'_{1}$ and the fact that $X_{-1} = X'_{-1}$, this shows that $X_1 = X'_1$
and $X_0 = X'_0$ as well. Finally, using~\eqref{eq:g-dv} again, we conclude
that $A = A'$.
\end{proof}

The bound in the first part of Proposition~\ref{Prop:cv-finite} is sharp: 

\begin{Expl}
The cyclic group of order $2n$, with the usual word metric of diameter $n$, 
satisfies $|\cC(v)| = 2n$ for every element $v$. 
The constant function $f$ with value $\frac12 n$ has $\rk(A(f)) = n$, 
and $A(f) \sub A(d_v)$ for all $v$. In fact, the injective
hull is a combinatorial $n$-cube, as is shown in~\cite[Section~9]{GooM}.
\end{Expl}

We now turn to discretely geodesic metric spaces $X$ with 
$\beta$-stable intervals, as defined in~\eqref{eq:stable}.
The following observation goes back to Cannon~\cite[Section~7]{Can}. 
For $x,v \in X$, define 
$F_{xv} \colon B(v,\beta) \to \Z$ by $F_{xv}(u) := d_x(u) - d_x(v)$.

\begin{Lem} \label{Lem:cone-type}
Let $X$ be a discretely geodesic metric space with $\beta$-stable intervals,
and let $x,x',v \in X$. 
If $F_{xv} \le F_{x'v}$, then $\co(x,v) \sub \co(x',v)$. 
Hence, $F_{xv} = F_{x'v}$ implies that $\co(x,v) = \co(x',v)$.
\end{Lem}

In particular, if for a fixed vertex $v$ the closed 
ball $B(v,\beta)$ is finite, then there are only finitely many distinct 
such functions $F_{xv}$ as $x$ ranges over $X$ and so $|\cC(v)| < \infty$.

\begin{proof}
Suppose that $F_{xv} \le F_{x'v}$. We show by induction on $l \ge 0$ 
that every $y \in \co(x,v)$ with $d(v,y) = l$ is an element of $\co(x',v)$.
The case $l = 0$ is trivial, so let $y \in \co(x,v)$ with 
$d(v,y) = l \ge 1$. Choose a point $y' \in \bw(v,y)$ such that 
$d(v,y') = l - 1$; note that $y' \in \co(x,v)$. 
By the induction hypothesis, $y' \in \co(x',v)$ and thus $v \in \bw(x',y')$.
Since $d(y',y) = 1$ and $X$ has $\beta$-stable intervals, there exists 
a point $u \in \bw(x',y) \cap B(v,\beta)$. We have
\[
d_{x'}(u) - d_{x'}(v) = F_{x'v}(u) \ge F_{xv}(u) = d_x(u) - d_x(v).
\]
Adding the term $d(u,y) - d(v,y)$ and using the identities
$d_{x'}(u) + d(u,y) = d_{x'}(y)$ and $d_x(v) + d(v,y) = d_x(y)$ we obtain
\[
d_{x'}(y) - d_{x'}(v) - d(v,y) \ge d_x(u) + d(u,y) - d_x(y) \ge 0. 
\]
Thus $d_{x'}(y) = d_{x'}(v) + d(v,y)$ and so $y \in \co(x',v)$.
\end{proof}

Lemma~\ref{Lem:cone-type} shows that if a finitely generated group 
$\Gam_S = (\Gam,d_S)$ with the word metric has $\beta$-stable intervals,
then $|\cC(v)|$ is finite for every $v \in \Gam_S$, and this number is 
of course independent of $v$. The cones $\co(x,1)$ based at the identity
element of $\Gam$ will be called {\em cone types}. For groups with finitely 
many cone types the language of all geodesic words is regular and 
the growth series is a rational function (see~\cite{Can, Eps+}).

\begin{Rem} \label{Rem:fft-property}
Neumann--Shapiro~\cite{NeuS} introduced a similar criterion, 
the {\em falsification by fellow traveller} (FFT) property, 
which is easily seen to imply uniform stability of intervals.
In particular it follows from Proposition~4.4 and Theorem~4.3 in their paper
that all finitely generated abelian groups have $\beta$-stable intervals and 
that finitely generated virtually abelian groups as well as 
geometrically finite hyperbolic groups have $\beta$-stable intervals 
for {\em some} word metrics. 
The FFT~property has been verified for further classes of groups, with 
respect to suitably chosen finite generating sets, in~\cite{Nos,Nos2,Hol}.
\end{Rem}

We also remark that the uniform stability of intervals is not 
a necessary condition for a finitely generated group $\Gam_S$ 
to have finitely many cone types:

\begin{Expl}
The finitely presented group $\Gam = \la a,t \mid t^2 = 1,\,atat = tata \ra$
with generating set $S = \{a,t\}$ has finitely many
cone types, but intervals are not uniformly stable.
This example is discussed in~\cite{Eld}. 
\end{Expl}

The next result states another consequence of the stability 
assumption. For a pair of points $x,y$ in a metric space $X$,  
\begin{equation} \label{eq:Gromov-prd}
(x \mid y)_z := \frac12 \bigl( d_z(x) + d_z(y) - d(x,y) \bigr)
\end{equation}
denotes their Gromov product with respect to $z \in X$.
Note that $0 \le (x\mid y)_z \le \min\{d_z(x), d_z(y)\}$ and 
$(x\mid y)_z + (x \mid z)_y = d(y,z)$. 

\begin{Lem} \label{Lem:z-beta}
Let $X$ be a discretely geodesic metric space with $\beta$-stable intervals.
Whenever $x,y,z \in X$, there exists a point $v \in \bw(x,y)$ 
with $d_z(v) \le \beta \cdot 2(x \mid y)_z$.
\end{Lem}

\begin{proof}
We proceed by induction on the integer $2(x \mid y)_z$.
If $(x \mid y)_z = 0$, then $z \in \bw(x,y)$ and we can take $v = z$.
Now suppose that $(x \mid y)_z > 0$.
Choose a discrete geodesic $\gam \colon \{0,1,\dots,d_z(y)\} \to X$ from 
$z$ to $y$, and let $k$ be the largest parameter value such that 
$z \in \bw(x,\gam(k))$. Note that $k < d_z(y)$ because $(x \mid y)_z > 0$. 
Let $y' := \gam(k+1)$. Since $X$ has $\beta$-stable intervals,
there exists a point $z' \in \bw(x,y')$ with $d_z(z') \le \beta$. 
We have
\[
d_{z'}(x) + d_{z'}(y') = d(x,y') \le d_z(x) + d_z(y') - 1
\]
by the choice of $k$. Adding the term $d(y',y) - d(x,y)$ 
we obtain $2(x\mid y)_{z'} \le 2(x \mid y)_z - 1$.
Hence, by the induction hypothesis, there exists $v \in \bw(x,y)$
such that $d_{z'}(v) \le \beta \cdot 2(x \mid y)_{z'}$. So
\[
d_z(v) \le d_z(z') + d_{z'}(v) \le \beta \bigl( 1 + 2(x \mid y)_{z'} \bigr) 
\le \beta \cdot 2(x \mid y)_z,
\]
as desired.
\end{proof}

We conclude this section with a partial generalization
of Proposition~\ref{Prop:cv-finite}.  
The above lemma will be used in combination with the following 
simple fact: if $f \in \D{X}$ and $\{x,y\} \in A(f)$, that is,
$f(x) + f(y) = d(x,y)$, then
\begin{equation} \label{eq:xy-fz}
(x \mid y)_z 
\le \frac12 \bigl( (f(x) + f(z)) + (f(y) + f(z)) - d(x,y) \bigr) = f(z)
\end{equation}
for every $z \in X$. For a subset $B$ of a metric space $X$ 
we denote by $\cC(B)$ the set of all pointed cones $(v,\co(x,v))$ with 
$v \in B$ and $x \in X$.

\begin{Prop} \label{Prop:loc-finite}
Let $X$ be a discretely geodesic metric space with $\beta$-stable 
intervals, and assume that all bounded subsets of $X$ are finite. 
Fix $z \in X$ and $\alpha > 0$, and let $B$ be the closed ball
$B(z,2\alpha\beta)$. Then $|\cC(B)| < \infty$, and 
\begin{enumerate}
\item[\rm (1)]
every $f \in \Ex{X}$ with $f(z) \le \alpha$ satisfies 
$\rk(A(f)) \le \frac12|\cC(B)|$;
\item[\rm (2)] 
for every $f \in \Ex{X}$ with $f(z) \le \alpha$ and 
$\rk(A(f)) = 0$ there are no more than $2^{|\cC(B)|}$ sets $A \in \cA(X)$ 
such that $A \sub A(f)$ and $\rk(A) = 1$.
\end{enumerate}
\end{Prop}

\begin{proof}
By the assumptions on~$X$ and Lemma~\ref{Lem:cone-type}, $\cC(B)$ is finite.
Now let $f \in \Ex{X}$, and suppose that $f(z) \le \alpha$.
For every ordered pair $(x,y)$ with $\{x,y\} \in A(f)$,
we choose a point $v_{xy} \in \bw(x,y) \cap B$
by means of Lemma~\ref{Lem:z-beta} and~\eqref{eq:xy-fz}, then we put
\[
\widehat C(x,y) := \bigl(v_{xy},\co(x,v_{xy})\bigr) \in \cC(B).
\]
Consider the partition of $X$ induced by $A(f)$, 
as in~\eqref{eq:x-partition}.
Let $x,\bar x \in X$, and choose $y,\bar y$ such that 
$\{x,y\},\{\bar x,\bar y\} \in A(f)$. 
If $\widehat C(x,y) = \widehat C(\bar x,\bar y)$,  
Lemma~\ref{Lem:xyxy} shows that $x$ and $\bar x$ are connected by an 
$A(f)$-path of length~$2$. Hence, if $x$ and $\bar x$
lie in different sets of the partition, then 
$\widehat C(x,y) \ne \widehat C(\bar x,\bar y)$. 
It follows that $\rk(A(f)) \le \frac12|\cC(B)|$.

Now suppose in addition that $\rk(A(f)) = 0$, and 
define $\widehat C(x,y) \in \cC(B)$ as above, for every pair $(x,y)$ 
with $\{x,y\} \in A(f)$.
Let $\cA_f'$ denote the set of all $A \in \cA(X)$ 
such that $A \sub A(f)$ and $\rk(A) = 1$. 
We show that there is an injective map $S$ from $\cA_f'$ into the set 
of all (non-empty) subsets of $\cC(B)$. 
For every $A \in \cA_f'$ there is a unique partition
$X = X_0 \cup X_1 \cup X_{-1}$ such that every $g \in \Ex{X}$ with $A(g) = A$
satisfies 
\[
g(x) = f(x) + \sig \|g - f\|_\infty
\]
for $\sig \in \{0,1,-1\}$ and $x \in X_\sig$. Define
\[
S(A) := \bigl\{ \widehat C(x,y) : 
(x,y) \in X_{-1} \times X_1,\,\{x,y\} \in A \bigr\};
\]
since $A \sub A(f)$, $\widehat C(x,y)$ is defined for all $(x,y)$ 
with $\{x,y\} \in A$. We claim that 
\[
X_{-1} = \bigl\{ x \in X : 
\text{$\widehat C(x,y) \in S(A)$ for all $y \in X$ with $\{x,y\} \in A(f)$} 
\bigr\}.
\]
Let $x \in X_{-1}$. If $y \in X$ is such that $\{x,y\} \in A(f)$,
then every $g \in \Ex{X}$ with $A(g) = A$ satisfies
$g(y) \ge d(x,y) - g(x) = f(x) + f(y) - g(x) = 
f(y) + \|g - f\|_\infty$, so $y \in X_1$ and $\{x,y\} \in A(g) = A$; 
thus $\widehat C(x,y) \in S(A)$.  
Conversely, suppose that $\bar x \in X$, and  
$\widehat C(\bar x,\bar y) \in S(A)$ for all $\bar y \in X$ 
with $\{\bar x,\bar y\} \in A(f)$. Among all such points $\bar y$,
fix one with $\{\bar x,\bar y\} \in A$.
Since $\widehat C(\bar x,\bar y) \in S(A)$, there is a pair 
$(x,y) \in X_{-1} \times X_1$ such that 
$\{x,y\} \in A$ and $\widehat C(x,y) = \widehat C(\bar x,\bar y)$. 
By Lemma~\ref{Lem:xyxy}, $x$ and $\bar x$ 
are connected by an $A$-path of length~$2$, 
thus $\bar x \in X_{-1}$.
Hence, $X_{-1}$ is characterized in terms of $S(A)$ 
as claimed. Now one can proceed as in the proof of 
Proposition~\ref{Prop:cv-finite}, with $f$ in place of 
$d_v$, to show that $S$ is injective. This gives~(2).
\end{proof}


\section{Proofs of the main results} \label{Sect:proofs}

We now prove the results stated in the introduction 
and discuss some examples.

\begin{proof}[Proof of Theorem~\ref{Thm:intro-ex}]
Let $X$ be a discretely geodesic metric space with $\beta$-stable intervals, 
and suppose that all bounded subsets of $X$ are finite.
Given $f \in \E{X}$, there exists a point $z \in X$ where
$f$ attains its minimum. Fix any $\eps > 0$.
By the first part of Proposition~\ref{Prop:loc-finite} there exists a number 
$N$ such that $\rk(A(g)) \le N$ for all $g \in \Ex{X}$ with 
$g(z) \le f(z) + \eps$, in particular for all $g \in \Ex{X}$ with
$\|f - g\| < \eps$. Now Theorem~\ref{Thm:poly} shows that
$\Ex{X} = \E{X}$ and that $\cP = \{P(A)\}_{A \in \cA(X)}$ is a polyhedral 
structure on $\E{X}$ with locally finite dimension and with only finitely 
many isometry types of $n$-cells for every $n$. 
By Theorem~\ref{Thm:pa} every $n$-cell $P(A)$ is isometric to an 
injective polytope in $l_\infty^n$. To show that $\cP$ is in fact locally 
finite, let $f \in \E{X} = \Ex{X}$ be a vertex of $\cP$; that is,
$\rk(A(f)) = 0$. Let again $z$ be a point where $f$ attains the minimum,
and put $\alpha := f(z)$. By the second part of 
Proposition~\ref{Prop:loc-finite} there is a number $M$ such that
there are at most $M$ admissible sets $A \sub A(f)$ with $\rk(A) = 1$;
in other words, there are at most $M$ edges in $\cP$ issuing from the vertex
$f$. Thus $\cP$ is locally finite, and $\E{X}$ is locally compact.
Consequently, as a complete geodesic metric space, $\E{X}$ is proper.
\end{proof}

The following simple example shows that, with the assumptions of 
Theorem~\ref{Thm:intro-ex}, the injective hull may be 
infinite dimensional.

\begin{Expl} \label{Expl:infinite}
For every integer $n \ge 1$, $W_n := \{0,1\}^n$ with the
$l_1$ distance (the vertex set of the $n$-cube graph) has 
$1$-stable intervals, and the injective hull $\E{W_n}$ has dimension
$2^{n-1}$ (compare Example~\ref{Expl:zn}; see also~\cite[Section~5]{GooM}
for more precise information in the case $n = 3$). Now let $X$ be the 
space obtained from the disjoint union $\bigcup_{n = 1}^\infty W_n$ 
by identifying $(1,1,\dots,1) \in W_n$ with 
$(0,0,\dots,0) \in W_{n+1}$ for $n = 1,2,\dots$, equipped with the obvious 
discretely geodesic metric, so that $X$ contains an isometric copy of each 
$W_n$. Clearly $X$ has $1$-stable intervals, bounded subsets of $X$ are 
finite, and $\E{X}$ is infinite dimensional.
\end{Expl}

Next we establish Proposition~\ref{Prop:intro-fix}. 
Given a metric space $X$ and a group $\Lam$ of isometries of $X$, 
we write $\Lam x := \{L(x) : L \in \Lam\}$ for the orbit of $x$
and $\Lam\!\setminus\!X$ for the set of orbits; furthermore
$\Fix(\Lam) := \{x \in X : \Lam x = \{x\}\}$ denotes the fixed
point set.

\begin{proof}[Proof of Proposition~\ref{Prop:intro-fix}]
First we show that if $X$ is a metric space and $\Lam$ is a subgroup 
of the isometry group of $X$ with bounded orbits,
there exists an extremal function $f \in \E{X}$ that is constant on each 
orbit. For $\Lam x,\Lam y \in \Lam\!\setminus\!X$, define
\[
D(\Lam x,\Lam y) := \sup\{d(x',y') : x' \in \Lam x,\,y' \in \Lam y \}.
\]
Note that this is finite since the orbits are bounded, and $D$ has all the 
properties of a metric except that $D(\Lam x,\Lam x) = \diam(\Lam x) > 0$ 
if $x \not\in \Fix(\Lam)$. 
Denote by $\D{\Lam\!\setminus\!X,D}$ the set of all functions 
$G \colon \Lam\!\setminus\!X \to \R$ such that 
\[
G(\Lam x) + G(\Lam y) \ge D(\Lam x,\Lam y)
\] 
for all $\Lam x,\Lam y \in \Lam\!\setminus\!X$. For $z \in X$, 
the function defined by $G_z(\Lam x) := D(\Lam x,\Lam z)$ 
belongs to $\D{\Lam\!\setminus\!X,D}$, due to the triangle inequality for $D$. 
By Zorn's Lemma, the partially ordered set $(\D{\Lam\!\setminus\!X,D},\le)$ 
has a minimal element $F$. 
Consider the respective function $f \colon X \to \R$,
$f(x) := F(\Lam x)$. For all $x,y \in X$,
\[
f(x) + f(y) = F(\Lam x) + F(\Lam y) \ge D(\Lam x,\Lam y) \ge d(x,y),
\]
so $f \in \D{X}$. Furthermore, by the minimality of $F$, for every $x \in X$ 
and $\eps > 0$ there is a point $y \in X$ such that
$F(\Lam x) + F(\Lam y) \le D(\Lam x,\Lam y) + \eps$ and 
$D(\Lam x,\Lam y) \le d(x,y) + \eps$, 
hence $f(x) + f(y) \le d(x,y) + 2\eps$.
This shows that in fact $f \in \E{X}$.

Now suppose that $X$ is injective. Then the only extremal functions 
on $X$ are distance functions, so by the above result
there exists a point $z \in X$ such that $d_z$ is constant on each orbit 
of $\Lam$. Thus $\Lam z = \{z\}$ and so $z \in \Fix(\Lam)$.
We prove that $\Fix(\Lam)$ is hyperconvex 
(recall Proposition~\ref{Prop:hyperconvex}). 
Since $\Fix(\Lam) \ne \es$, it suffices to show that 
if $((x_i,r_i))_{i \in I}$ is a non-empty family in $X \times \R$
such that $x_i \in \Fix(\Lam)$ and $r_i + r_j \ge d(x_i,x_j)$ for all pairs of 
indices $i,j \in I$, then $Y := \bigcap_{i \in I}B(x_i,r_i)$ has non-empty 
intersection with $\Fix(\Lam)$.
Note that $Y$ is bounded and hyperconvex, in particular $Y \ne \es$.
For all $i \in I$, $L \in \Lam$, and $y \in Y$, we have 
\[
d(x_i,L(y)) = d(L(x_i),L(y)) = d(x_i,y) \le r_i,
\] 
thus $L(Y) \sub Y$. 
In other words, for every $L \in \Lam$, the restriction $L|_Y$ is an 
isometric embedding of $Y$ into itself. In fact, since also $L^{-1}(Y) \sub Y$,
$L|_Y$ is an isometry of $Y$. Since $Y$ is bounded and injective, 
the group $\{L|_Y : L \in \Lam\}$ must have a fixed point, 
as we already know, so $Y \cap \Fix(\Lam) \ne \es$.
\end{proof} 

We proceed to $\del$-hyperbolic metric spaces, as defined in~\eqref{eq:hyp}.

\begin{proof}[Proof of Proposition~\ref{Prop:intro-hyp}]
To show that $\E{X}$ is $\del$-hyperbolic, let $e,f,g,h \in \E{X}$, 
and let $\eps > 0$. There exist $w,x \in X$ such
that either $\|e - f\|_\infty \le e(x) - f(x) + \eps$ and 
$e(x) \le d(w,x) - e(w) + \eps$, or $\|e - f\|_\infty \le f(w) - e(w) + \eps$
and $f(w) \le d(w,x) - f(x) + \eps$. Thus, in either case,
\[
\|e - f\|_\infty \le d(w,x) - e(w) - f(x) + 2\eps.
\]
Likewise, $\|g - h\|_\infty \le d(y,z) - g(y) - h(z) + 2\eps$ for some
$y,z \in X$. Put $\Sig := e(w) + f(x) + g(y) + h(z)$.
Using the $\del$-hyperbolicity of $X$ we obtain
\begin{align*}
&\|e - f\|_\infty + \|g - h\|_\infty \le d(w,x) + d(y,z) - \Sig + 4\eps \\
&\qquad\le \max\{d(w,y) + d(x,z),d(x,y) + d(w,z)\} - \Sig + \del + 4\eps.
\end{align*}
Now $d(w,y) + d(x,z) - \Sig \le e(y) + f(z) - g(y) - h(z) \le \|e-g\|_\infty + 
\|f-h\|_\infty$ and
$d(x,y) + d(w,z) - \Sig \le -e(w) - f(x) + g(x) + h(w) \le \|f-g\|_\infty + 
\|e-h\|_\infty$. Since $\eps > 0$ was arbitrary, this gives the desired 
inequality for $e,f,g,h$. 

Suppose, in addition, that $X$ is geodesic or discretely geodesic.
Put $\nu := 0$ in the former and $\nu := \frac12$ in the latter case.
Let $f \in \E{X}$. For $\eps > 0$, choose $x,y \in X$ such that 
$f(x) + f(y) \le d(x,y) + \eps$. Since $f(x) + f(y) \ge d(x,y)$,
there is a point $v \in \bw(x,y)$ such that $d(v,x) \le f(x) + \nu$ 
and $d(v,y) \le f(y) + \nu$.
Using the $\del$-hyperbolicity of the quadruple 
$\{f,d_v,d_x,d_y\} \sub \E{X}$, together with~\eqref{eq:fdf}, we get
\begin{align*}
f(v) + d(x,y) 
&\le \max\{ f(x) + d(v,y), f(y) + d(v,x) \} + \del \\
&\le f(x) + f(y) + \del + \nu \\
&\le d(x,y) + \eps + \del + \nu,
\end{align*}
thus $f(v) \le \eps + \del + \nu$. Hence, for every $\eps > 0$ there exists 
$v \in X$ such that $\|f - d_v\|_\infty = f(v) \le \eps + \del + \nu$.
\end{proof}

Now let $\Gam_S = (\Gam,d_S)$ be a finitely generated group with the word 
metric. By Proposition~\ref{Prop:isometries}, the isometric action 
$(x,y) \mapsto L_x(y) := xy$ of $\Gam$ on $\Gam_S$ induces an 
isometric action 
\[
(x,f) \mapsto \bar{L}_x(f) = f \circ L_x^{-1}
\] 
of $\Gam$ on the injective hull $\E{\Gam_S}$. Suppose that $\Gam_S$ has 
$\beta$-stable intervals. Then we already know 
that $\E{\Gam_S}$ is proper and that $\cP = \{P(A)\}_{A \in \cA(\Gam_S)}$ 
is a locally finite polyhedral structure on $\E{\Gam_S}$ with finitely 
many isometry types of $n$-cells for every $n$.

Before we proceed with the proof of Theorem~\ref{Thm:intro-groups},
we recall the construction of the first barycentric subdivision
of $\E{\Gam_S}$.
The barycenter $b$ of a finite family $(v_i)_{i=1}^k$
of points in a vector space $Z$ is defined as $b := \frac1k\sum_{i=1}^k v_i$
or, equivalently, as the unique point $b$ such that $\sum_{i=1}^k(v_i-b) = 0$.
If $L \colon Z \to Z$ is an affine map, the barycenter of 
$(L(v_i))_{i=1}^k$ equals $L(b)$. If, in addition, $L(v_i) = v_{\sig(i)}$
for some permutation $\sig$ of $\{1,\dots,k\}$, then $L(b) = b$.
Now, for every cell $P(A)$ of $\cP$, 
the barycenter $b(A)$ of $P(A)$ is defined as the barycenter in $\R^\Gam$ of 
the vertex set of $P(A)$. Clearly $b(A)$ is a point in the interior of 
$P(A)$ relative to the affine hull $H(A)$ of $P(A)$. Every isometry $L$ of 
$P(A)$ is the restriction of an affine transformation of $H(A)$ that permutes 
the vertices of $P(A)$, so $L(b(A)) = b(A)$. The first barycentric 
subdivision $\cP^1$ of $\cP$ is the collection of all simplices 
$[b(A_0),b(A_1),\dots,b(A_j)]$ corresponding to strictly ascending 
sequences $P(A_0) \sub P(A_1) \sub \ldots \sub P(A_j)$ of cells in $\cP$. 
We write $\E{\Gam_S}^1$ for the metric space $\E{\Gam_S}$ 
equipped with the simplicial structure $\cP^1$.

\begin{proof}[Proof of Theorem~\ref{Thm:intro-groups}]
First we show that for every bounded set $B \sub \E{\Gam_S}$ 
there are only finitely many $x \in \Gam$ such that 
$\bar{L}_x(B) \cap B \ne \es$. Let $R > 0$ be such that 
$\|f - d_1\|_\infty \le R$ for all $f \in B$,
where $d_1$ is the distance function to $1 \in \Gam$.
We have $\bar{L}_x(d_1) = d_x$. Hence, if $f \in \bar{L}_x(B) \cap B$,
then also $\|f - d_x\|_\infty \le R$ and so 
$d_S(1,x) = \|d_1 - d_x\|_\infty \le 2R$. 
This gives the result. As this holds for compact sets $B$, 
the action of $\Gam$ on $\E{\Gam_S}$ is proper (recall also that 
$\E{\Gam_S}$ is itself proper).
For $x \in \Gam$ and $f,g \in \E{\Gam_S}$, 
it follows from the left-invariance of $d_S$ that $A(f) = A(g)$ if and only if 
$A(f \circ L_x^{-1}) = A(g \circ L_x^{-1})$, which is in turn equivalent to 
$A(\bar{L}_x(f)) = A(\bar{L}_x(g))$. So $\bar L_x$ maps cells in $\cP$
onto cells.

If we pass to the first barycentric subdivision of $\E{\Gam_S}$, the 
group $\Gam$ still acts by cellular---now simplicial---isometries 
on $\E{\Gam_S}^1$ via $x \mapsto \bar L_x$. In addition, if $\bar L_x$ maps 
a simplex in $\cP^1$ to itself, then $\bar L_x$ fixes the simplex pointwise.
Thus $\E{\Gam_S}^1$ is a $\Gam$-CW-complex. For every finite 
subgroup $\Lam$ of $\Gam$, the fixed point subcomplex $\Fix(\Lam)$ 
is contractible by Proposition~\ref{Prop:intro-fix}. Since $\Gam$ acts 
properly, this shows that $\E{\Gam_S}^1$ is a model for the classifying space
$\underbar{\rm E}\Gam$ for proper $\Gam$-actions (see~\cite[Section~1]{Lue}).

Now suppose that $\Gam_S$ is $\del$-hyperbolic.
It follows from Proposition~\ref{Prop:intro-hyp} that for every 
$f \in \E{\Gam_S}$ there is a point $z \in \Gam_S$ with 
$\|f - d_z\| \le \del + \frac12$. Hence
the closed ball in $\E{\Gam_S}$ with center $d_1$ and radius 
$\del + \frac12$ is a compact set whose $\Gam$-orbit covers $\E{\Gam_S}$.
As in the second half of the proof of Proposition~\ref{Prop:intro-hyp} we see
that whenever $f \in \E{\Gam_S}$ and $\{x,y\} \in A(f)$,
there exists a point $v_{xy} \in \bw(x,y)$ with 
$f(v_{xy}) \le \del + \frac12$. If $\{x',y'\}$ is another element of $A(f)$,
then $d(v_{xy},v_{x'y'}) \le f(v_{xy}) + f(v_{x'y'}) \le 2\del + 1$.  
The argument of the first part of Proposition~\ref{Prop:loc-finite}
then shows that
\begin{equation} \label{eq:dim-bound}
\rk(A(f)) \le \frac12 \cdot
\max\{|B| : B \sub \Gam_S,\,\diam(B) \le 2\del + 1\}
\cdot |\cC(1)|
\end{equation}
for all $f \in \E{\Gam_S}$,
where $|\cC(1)|$ is the number of cone types of $(\Gam,S)$.
So the dimension of $\E{\Gam_S}$ is bounded by the right side 
of~\eqref{eq:dim-bound} too.
\end{proof}

In order for the injective hull $\E{\Gam_S}$
to lie within finite distance of $\e(\Gam_S)$, $\Gam_S$ need not be 
word hyperbolic, as is shown by $\Z^2$ (compare Example~\ref{Expl:zn}). 
A necessary condition is given next.

\begin{Rem}
Let $\Gam_S = (\Gam,d_S)$ be a finitely generated group with the word
metric, and suppose that there is a constant $D$ such that for every 
$f \in \E{\Gam_S}$ there is an element $z \in \Gam_S$ with 
$\|f - d_z\|_\infty = f(z) \le D$. Note that $\Gam$ acts coboundedly 
on $\E{\Gam_S}$. It follows from Proposition~\ref{Prop:bicombing} that
there is map $\sig \colon \Gam_S \times \Gam_S \times [0,1] \to \Gam_S$ 
with the following properties: for every pair $(x,y) \in \Gam_S \times \Gam_S$,
the map $\sig_{xy} := \sig(x,y,\cdot)$ satisfies $\sig_{xy}(0) = x$, 
$\sig_{xy}(1) = y$, and $|d_S(\sig_{xy}(s),\sig_{xy}(t)) - (t-s)d_S(x,y)| \le 2D$ 
for $0 \le s \le t \le 1$; furthermore, 
\[
d_S(\sig_{xy}(t),\sig_{x'y'}(t)) \le (1-t)d(x,x') + td(y,y') + 2D
\]
and $z \cdot \sig_{xy}(t) = \sig_{zx,zy}(t)$ 
for $x,y,x',y' \in \Gam_S$, $t \in [0,1]$, and $z \in \Gam$.
In particular, $\Gam_S$ is semihyperbolic in the sense of 
Alonso--Bridson~\cite{AloB}. On the other hand, $\Z^n$ for $n \ge 3$ 
is an example of a semihyperbolic group that does not act coboundedly 
on its injective hull.
\end{Rem}

For a finitely generated group $\Gam_S$ with $\beta$-stable intervals that 
is not word hyperbolic, Theorem~\ref{Thm:intro-groups} leaves open 
the possibility of $\E{\Gam_S}$ being infinite dimensional. 
An example of this type is missing at present. 
However, there are simple instances of finitely presented groups 
(without uniformly stable intervals) whose injective hull 
fails to be finite dimensional or locally compact. 
We also note that groups with $\beta$-stable intervals are easily 
seen to be almost convex in the sense of Cannon~\cite{Can2} 
and hence finitely presented.

\begin{Expl}
Let $\Gam_S$ be the Baumslag--Solitar group 
$\la x,y \mid yx = x^2y \ra$ with
generating set $S = \{x,y\}$.
Fix an integer $n \ge 1$. The word
$w_n := u_nxu_n^{-1}x^{-1}$, where $u_n := y^nx^2y^{-n}$,
represents the identity. Let $\gam \colon \{0,1,\dots,l\}
\to \Gam_S$ be the corresponding discrete loop of length $l := 4n + 6$.
This is similar to the loop depicted 
in~\cite[Figure~7.8]{Eps+} (where $u_n$ is chosen to be $y^nxy^{-n}$).
By inspecting this picture, one sees that for $k = 0,\dots,n$,
the two points $\gam(k) = y^k$ and $\gam(\frac12 l + k) = u_nxy^k =
x^{2^{n+1} + 1}y^k$ are at distance $\frac12 l$ from each other.
It follows that the constant function $f = \frac14 l$
on $Y := \bigcup_{k=0}^n\{\gam(k),\gam(\frac12 l + k)\}$
is an element of $\E{Y}$ with $\rk(A(f)) = n + 1$.
As $n \ge 1$ was arbitrary, $\E{\Gam_S}$ cannot be finite dimensional.
\end{Expl}

The following example shows that for a finitely presented group $\Gam_S$
with infinitely many cone types, $\E{\Gam_S}$ need not be locally finite
near points in $\e(\Gam_S)$. This contrasts with the second assertion of 
Proposition~\ref{Prop:cv-finite}. 

\begin{Expl} \label{Expl:not-loc-cpt}
Consider the group $\Gam = \la a,b,t \mid ab = ba,\, t^2=1,\, tab = abt \ra$
with generating set $S = \{a,b,t\}$.
For every integer $m \ge 1$, put $x_m := a^{-m}t$ and $y_m := b^mt$.
Note that $d_S(1,x_m) = d_S(1,y_m) = m + 1$, 
\[
d_S(x_m,y_m) = d_S(1,ta^mb^mt) = d_S(1,t(ab)^mt) = 2m
\] 
for all $m \ge 1$, and $d_S(x_m,y_n) = m + n + 2$ if $n \ne m$.
Hence, $\co(x_m,1)$ contains $\{y_n : n \ne m\}$ but 
not $y_m$, so $\Gam_S$ has infinitely many cone types.
Now let $f_m \in \E{\Gam_S}$ be a median point of the triple
of distance functions $d_1,d_{x_m},d_{y_m}$. Then $\|f_m - d_1\|_\infty = 1$
and $\|f_m - d_{x_m}\|_\infty = \|f_m - d_{y_m}\|_\infty = m$, 
and it follows from the triangle inequality that $\|f_n - f_m\|_\infty = 2$
whenever $n \ne m$. Hence, there is an isometrically embedded simplicial 
tree with infinite valence at the vertex $d_1$, which therefore has 
no compact neighborhood in $\E{\Gam_S}$.
Nevertheless, I suspect $\E{\Gam_S}$ to be a polyhedral complex of
finite dimension (equal to~3).
\end{Expl}

\bigskip
{\bf Acknowledgements.} 
I thank Mario Bonk, Arvin Moezzi, Pierre Pansu, and Viktor Schroeder 
for inspiring discussions and valuable comments. 
Parts of this paper were written 
during visits to the Institut Henri Poincar\'{e} in Paris, the Max Planck
Institute for Mathematics in Bonn, and the University of Seville. 
I gratefully acknowledge support from these institutions and from the Swiss
National Science Foundation.



\end{document}